\documentclass[oneside,final,11pt]{PKPZU_Math_class}

\usepackage{float} 
\usepackage[backend=biber,style=alphabetic,]{biblatex}

\usepackage[latin1]{inputenc}

\usepackage{cleveref}
    
\usepackage{thmtools}
\usepackage{thm-restate}

\usepackage{fancyhdr} 
\hypersetup{
    colorlinks=true,
    linkcolor=blue,
    filecolor=magenta,      
    urlcolor=cyan,
    pdftitle={Periodic TASEP and Young tableaus},
    pdfpagemode=FullScreen,
    }
\usepackage{blindtext}

\title{Periodic TASEP and a Cylindrical dual-RSK}
\author{Tomas Kojar\\	Michigan state \\ \url{kojartom@msu.edu}
}
\begingroup
\lccode`\~=`\ %
\lowercase{%
  }%
\endgroup
\newcounter{word}

\makeatletter
\newenvironment{proofs}[1][\proofname]{\par
  \pushQED{\qed}%
  \normalfont \topsep0\p@\relax
  \trivlist
  \item[\hskip\labelsep\itshape
  #1\@addpunct{.}]\ignorespaces
}{%
  \popQED\endtrivlist\@endpefalse
}
\makeatother

\usepackage{ytableau}

\makeatletter
\let\old@rule\@rule
\def\@rule[#1]#2#3{\textcolor{blue}{\old@rule[#1]{#2}{#3}}}
\makeatother

\theoremstyle{definition}

\newtheorem{example}{Example}
\AtBeginEnvironment{example}{%
  \pushQED{\qed}%
}
\AtEndEnvironment{example}{\popQED\endexample}

\begin{document}
\maketitle
\begin{abstract}
We introduce a version of RSK that captures the dynamics of periodic TASEP. This work is building towards a cylindrical analogue of \parencite{DW08}.    
\end{abstract}
\tableofcontents

\section{Introduction}
In \parencite{schutz1997exact} \Schutz obtained a determinant formula for $\ZTASEP$
\begin{equation}
P(X(t)=x|X(0)=y)=\det[(F_{i-j}(x_{N+1-i}-y_{j};t))_{1\leq i,j\leq N}]   ,    
\end{equation}
that became the basis for later constructing the KPZ fixed point \cite{matetski2021kpz}. Here is the kernel in contour form
\begin{equation}\label{eq:contourformSchutz}
F_{i-j}(x_{N+1-i}-y_{N+1-j};t):= \oint_{|w|=R}e^{t(w-1)}\frac{(w-1)^{-(i-j)}}{w^{x_{N+1-i}-y_{N+1-j}-(i-j)+1}}\frac{dw}{2\pi i},
\end{equation}
with $0\leq y_{1}<y_{2}<...<y_{N}$ and $0\leq x_{N}<x_{N-1}<...<x_{1}$ and in summation form 
\begin{eqalign}
F_{p}(n;t):=&e^{-t}\sumls{k=0}^{\infty}\binom{k+p-1}{p-1}\frac{t^{k+n}}{(k+n)!}=\branchmat{e^{-t}\sumls{k=0}^{\infty}\binom{k+p-1}{p-1}\frac{t^{k+n}}{(k+n)!}&\tcwhen p>0\\ e^{-t}\sumls{k=0}^{\abs{p}}(-1)^{k}\binom{\abs{p}}{k}\frac{t^{k+n}}{(k+n)!} &\tcwhen p\leq 0},    
\end{eqalign}
for $\binom{a}{b}=\frac{\Gamma(a+1)}{\Gamma(b+1)\Gamma(a-b+1)}$. This formula was obtained via the use of Bethe ansatz. In the work \parencite{DW08}, this formula was rederived using the dynamics of RSK and formulas based on flagged Schur functions (in \cref{prop:continuoustime}, we include the limit computation). In this work, we aim to begin the process of building the cylindrical analogue of \parencite{DW08}  for periodic TASEP of N particles and period L. In particular, we aim towards providing a result of the form
\begin{eqalign}\label{eq:contourformSchutzcircle}
P(X(t)=x|X(0)=y)=&\lim_{n\to +\infty}\Pi\ast P_{\floor{nt}}\ast \Lambda(x,y) \\
=&\sum_{\substack{\sum k_{i}=0\\(k_{1},..,k_{N})\in \Z^{N} }}\det[(F_{i-j-Nk_{i}}(x_{N+1-i}-y_{j}- Lk_{i};t))_{1\leq i,j\leq N}]   ,    
\end{eqalign}
where $\Pi,P_{\floor{nt}}, \Lambda$ are summations over some versions of cylindrical skew diagrams and  $0\leq y_{1}<y_{2}<...<y_{N}\leq L-1$ and $0\leq x_{N}<x_{N-1}<...<x_{1}\leq L-1$. The particle-unlabeled version of this formula already shows up in \parencite{priezzhev2003exact}. The step from the operators to the sum requires the ideas from \cite[proposition 1]{gessel1997cylindric}. One ultimate hope is that one can find the cylindrical analogues of \cite{BFPS},\cite{matetski2021kpz} and apply it to each of the operators $\Pi,P_{\floor{nt}}, \Lambda$ (see \cite{huh2023bounded} for a version of cylindrical Cauchy Binet identity) and obtain a limiting formula of the \textit{periodic KPZ fixed point} in terms of their convolution. Finding a version of \cite{BFPS} for the summation formula seems to have serious fundamental challenges, some of which are described in \cref{sec:frameworkchallenges}.

\begin{remark}
The formula \cref{eq:contourformSchutzcircle} is also analogous to the situation with the heat equation with periodic boundaries $[-L,L]$ and the real line. The periodic heat kernel is
\begin{eqalign}
H_{t}^{L}(x)=&\frac{1}{\sqrt{4\pi t/L}}\sum_{k\in{\mathbb Z}}\expo{-\frac{(\frac{x}{L}-k)^2}{4t/L}}=&\sum_{k\in{\mathbb Z}} \expo{-\frac{4\pi^2}{L^{2}} k^2 t+2\pi i k \frac{x}{L}},   
\end{eqalign}
where here too we can see the presence of the real line $H_{t}^{\infty}(x):=\frac{1}{\sqrt{4\pi t}}\expo{-\frac{x^2}{4t}}$ shifted by $k$.
\end{remark}
\begin{remark}
 The Toda lattice model is another situation where we have a nice overarching formula for periodic and infinite line, see \parencite{krichever2000periodic} where due to the periodicity the authors are also forced to work with infinite dimensional objects (in their Lax pair choice) just as we are here having to work with all possible shifts according to vector $\mathbf{k}\in{\mathbb Z}^{N}$.    
\end{remark}
\subsection{The Priezzhev-formula \cref{eq:contourformSchutzcircle} and the Baik-Liu formula }
In \parencite[proposition 5.1]{baik2018fluctuations}, we have the transition formula for  $\PTASEP$ with a finite number $N$ of particles
\begin{equation} 
\label{eq:Betheroots}
P(X(t)=x|X(0)=y)=\oint_{|z|=r} \det\spara{\frac{1}{L}\sum_{w\in R_{z}}\frac{f(w)}{q'_{z}(w)}  }_{1\leq i,j\leq N}\frac{\dz}{2\pi i z},    
\end{equation}
where 
\begin{eqalign}
&f(w):=(w-1)^{j-i+N}w^{-x_{i}+y_{j}+i-j+L-N-1 }e^{t(w-1)}\tand q_{z}(w):=p(w)-z^{L}:=(w-1)^{N}w^{L-N}-z^{L}, \\
&R_{z}:=\set{w\in \mathbb{C}:q_{z}(w)=0},
\end{eqalign}
and the integral is over any simple closed contour in $\abs{z}>0$ that contains $0$ inside and $\rho:=\frac{N}{L}$. We show here yet another proof of the \cref{eq:contourformSchutzcircle}.
\begin{lemma}\label{lem:RintformulaPTASEP}
The formula \cref{eq:Betheroots} implies \cref{eq:contourformSchutzcircle}.
\end{lemma}
\begin{proof}
The strategy is to use the geometric series and the residue theorem to return back to \Schutz's formula for $\ZTASEP$. The function 
\begin{equation}
f(w) =(w-1)^{j-i+N}w^{-x_{i}+y_{j}+i-j+L-N-1 }e^{t(w-1)}
\end{equation}
 is analytic on $\mathbb{C}\setminus \set{0}$ and so by residue
\begin{equation}
\frac{1}{L}\sum_{w\in R_{z}}\frac{f(w)}{q'_{z}(w)}=\oint_{|w|=R}\frac{f(w)}{p(w)-z^{L}} \frac{dw}{2\pi i}-\oint_{|w|=\frac{1}{R}}\frac{f(w)}{p(w)-z^{L}} \frac{dw}{2\pi i},
\end{equation}
where $R$ is large enough so that the roots of $q_{z}(w)$ are contained in the annulus $A(\frac{1}{R},R)$. Next we use the geometric series for each of the contours. For the first contour we take R large enough so that $\abs{\frac{z^{L}}{p(w)}}<1$ and obtain
\begin{equation}
   \oint_{|w|=R}\frac{f(w)}{p(w)-z^{L}} \frac{dw}{2\pi i} =\sum_{k\geq 0}z^{L k}\oint_{|w|=R}\frac{f(w)}{p(w)}\frac{1}{p(w)^{ k}}\frac{dw}{2\pi i}
\end{equation}
and for the second contour we take $R$ large enough so that $\abs{\frac{p(w)}{z^{L}}}<1$ and obtain
\begin{equation}
  - \oint_{|w|=\frac{1}{R}}\frac{f(w)}{p(w)-z^{L}} \frac{dw}{2\pi i} =\sum_{k\leq -1 }z^{L k}\oint_{|w|=\frac{1}{R}}\frac{f(w)}{p(w)}\frac{1}{p(w)^{ k}}\frac{dw}{2\pi i}.
\end{equation}
For the second contour since $k+1\leq 0$, the only pole comes from $w=0$ and so we can enlarge the contour to match the first one
\begin{equation}\label{eq:Lt1}
  \sum_{k\leq -1 }z^{L k}\oint_{|w|=R}\frac{f(w)}{p(w)^{1+ k}}\frac{dw}{2\pi i}
\end{equation}
and then just add the two
\begin{equation}\label{eq:Lt2}
 \sum_{k\in \Z }z^{L k}\oint_{|w|=R}\frac{f(w)}{p(w)^{k+1}}\frac{dw}{2\pi i}.
\end{equation}
Finally by pulling out the sums from each row, we obtain
\begin{equation}
  \sum_{k_{1},...,k_{N}\in \Z }  z^{L\sum k_{i}-1}.  
\end{equation}
We obtain that for the z-contour to be non-zero we require $L(\sum k_{i}-\frac{a_{i}}{N})=0$. 

\end{proof}
\paragraph{Some literature on cylindrical versions of RSK}Based on the cylindric tableau introduced in \parencite{postnikov2005affine}, in the work \parencite{neyman2014cylindric}, they introduce a cylindrical version of RSK. In \parencite{elizaldewalks} they even prove a bijection of these tableaus with the periodic TASEP process.


\paragraph{Acknowledgements}
We firstmost thank K.Matetski for suggesting the problem of deriving the formula \cref{eq:contourformSchutzcircle} using the framework of \parencite{DW08}. We thank D.Grinberg for his help with proving an inversion identity using posets. We also thank P.Alexandersson, C.Krattenthaler and K.Motegi for discussing this problem.

\section{Cylindrical Dual-RSK}
\subsection{Dual RSK }
We quickly recall the Dual-RSK definition from \parencite{DW08}, in particular the method under \textit{case B} because that is the one that corresponds to $\ZTASEP$. We start from the innovation data
\begin{equation}\label{eq:innovationdata}
\bigl(\begin{smallmatrix}
a_{1}& a_{2}& \cdots \\
b_{1}& b_{2}& \cdots \\
\end{smallmatrix}\bigr),     
\end{equation}
where the $a_{i},b_{i}$ are nondecreasing (lexicograhic). In the context of $\ZTASEP$ and PTASEP, the arrays keep track of the exponential clocks hitting but none of the particle exclusion interactions: two consecutive particles bumping into each other and the first particle hitting against the last particle after winding around. We consider the space arrays of strictly positive integers
\begin{equation}
\T_{N}^{n, \wedge }:=\set{T=(T_{ij}),i\in [N], j\in [\lambda_{i}]: T_{ij}\in [n], T_{ij}< T_{i+1,j},T_{ij}\leq T_{i,j+1},\lambda_{i}\geq 1   },    
\end{equation}
where the vector $(\lambda_{1},...,\lambda_{N})$ is called the shape of $T$ and denoted as $sh(T)$. In the terminology of enumerative combinatorics, $\T_{N}^{n, \wedge }$ consists of semi-standard Young tableaux (SSYT) with at most $N$ rows and content $\set{1,...,n}$. Similarly, we write $\T_{N}^{n, < }$ if the integers in $T$ do not exceed $n$ and increase strictly along the rows and we have at most $N$ rows.
\begin{definition}(Dual RSK)
Given $\CP(n)$, the SSYT $\CP(n+1)$ is found by inserting the elements $b_i$ for which $a_i = n + 1$.  If there are $M$ such elements, the methods construct a sequence $\CP^1(n),...,\CP^M(n)$ such that $\CP^1(n) = \CP(n) \tand \CP(n + 1) = \CP^M(n)$.
\begin{enumerate}
    \item The entries $b_i$ are inserted down the columns of $\CP(n)$.

    \item If every entry in the column $C$ is strictly smaller than $b$, then $b$ is appended to the end of the column.
    
    \item Otherwise, $b$ is used to replace the uppermost entry $w$ in the column $C$ which is \textit{greater than or equal} to $b\leq w$. 
    
    \item And we restart the process for $w$ by inserting it in the next column $C+1$,and this process continues until an entry is either placed at the end of a column or it is placed in the first position of an empty column.
    
\end{enumerate}
Let $\CQ(n)$ the unique array of integers for which the entries $1,...,m $ form an array with the same shape as $\CP(m)$ for $m \leq n$. The pair $(\CP(n),\CQ(n))$ belongs to
\begin{equation}
\set{(P,Q)\in \T_{N}^{n, \wedge }\times \T_{N}^{n, <}:  sh(P) = sh(Q )   }.    
\end{equation}
\end{definition}
\begin{example}
Consider the innovation data
\begin{equation}
\bigl(\begin{smallmatrix}
1 &1  &1  &2  &2   \\ 
1& 3 & 4 & 2 & 3
\end{smallmatrix}\bigr).     
\end{equation}
Then we have
\begin{eqalign}
&\CP(1)=\ytableausetup{textmode} \begin{ytableau}1\\3\\4 \end{ytableau},\CQ(1)=\ytableausetup{textmode} \begin{ytableau}1\\1\\1 \end{ytableau}\\
&\CP(2)=\ytableausetup{textmode} \begin{ytableau}1&3\\2&4\\3 \end{ytableau},\CQ(2)=\ytableausetup{textmode} \begin{ytableau}1&2\\1&2\\1 \end{ytableau} 
\end{eqalign}
with final shape $(2,2,1)$. 
\end{example}
This combinatorial correspondence is in bijection to the innovation data \cref{eq:innovationdata} (eg. see \parencite[Theorem 7.11.5,Theorem 7.14.1]{stanley2011enumerative}).
\pparagraph{Insertion paths}
Because we will need to prove a bijection for the cylindrical case, we recall the notion of \textit{insertion paths/routes} $I(k\to \CP)$ \parencite[lemma 7.11.2]{stanley2011enumerative}. This contains the coordinates of all the boxes that got bumped after we inserted the entry $k$ into the tableau $\CP$. In the above example, we have
\begin{eqalign}
I(2\to \CP(1))=\set{(2,1),(1,2) }, 
\end{eqalign}
because it displaced the box with entry $3$ to the second column. 
\begin{definition}
We will use the following operations. When each coordinate in the insertion path $R$ is located strictly north of the corresponding column-coordinate of route $R'$, we say $R$ is \textit{strictly above} $R'$ or $R'$ is \textit{strictly below} $R$ and write
\begin{eqalign}
R'\prec    R.
\end{eqalign}
Similarly if we also some of the y-coordinates to match, we replace strictly by weakly and write
\begin{eqalign}
R'\preceq    R.
\end{eqalign}
\end{definition}
We have the following \textit{column bumping lemma} as stated in \parencite[exercise 3 "Column bumping lemma" section A.2 p.g. 186]{fulton1997young} and proved analogously as \parencite[lemma 7.11.2]{stanley2011enumerative} and \parencite[section 1.1 "Row bumping lemma"]{fulton1997young}.
\begin{lemma}\label{lem:insertionpathproperties}
Fix SYYT $\CP$. Consider two successive column-insertions, first column-inserting $x$ in $\CP$ and then column-inserting $x'$ in $\CP'=(x\to\CP)$, giving rise to insertions-paths $R,R'$ and two new boxes with locations $B,B'$ respectively.
\begin{enumerate}

\item If $x< x'$, then $R'\prec R$, and $B'$ is southwest of $B$.

\item If $x\geq x'$, then $R \preceq  R'$, and $B'$ is northeast of $B$.

    
\end{enumerate}

\end{lemma}

In the cylindrical case we will need to deal with multiple insertions and so we state here a lemma that still fits in the Dual-RSK case.
\begin{lemma}\label{lem:intertwinninglemmathreeinsertions}
Fix SYYT $\CP$. Consider three successive column-insertions of entries $a_{1},a_{2}$ and $b$ with
\begin{eqalign}
a_{1}<b<a_{2}    
\end{eqalign}
i.e. we first column-inserting $a_{1}$ in $\CP_{1}$, we then column-insert $a_{2}$ in $\CP_{2}=(a_{1}\to\CP_{1})$, and then finally $b$ in $\CP_{3}=(a_{2}\to\CP_{2})$. We denote their routes $R_{1},R_{2},R_{3}$ with their corresponding new boxes with locations $B_{1},B_{2},B_{3}$ respectively. Then
\begin{itemize}
    \item $R_{3}$ is strictly below $R_{1}$ and weakly above $R_{2}$ i.e. $R_{2}\prec R_{3}\prec R_{1}$.
    \item $B_{3}$ is southwest of $B_{1}$ and northeast of $B_{2}$. 
\end{itemize}

\end{lemma}
    

\subsection{Cylindrical Infinite Young tableau (CIYT) }

\begin{definition}(Cylindrical Infinite Young tableau (CIYT))
We fix positive integers $N,L$. We consider a bi-infinite series of copied Young tableaus of height $N$ (i.e. each column has at most size $N$). The \textit{main period} tableau $\CP_{0}$ is filled with entries that are positive and the smallest. Note that $\CP_{0}$'s entries can move past the Nth row e.g. when there is an overflow of boxes. The shift between each tableau is $L-N$. The entries of tableaus above and below $\CP_{0}$ are shifted by $N$ and $-N$ respectively. We denote the \textit{winding} of each tableau by a subscript increasing for the tableaus $\CP_{1},\CP_{2},...$ above $\CP_{0}$ and decreasing for the the tableaus $\CP_{-1},\CP_{-2},...$ below it. 
\end{definition}
\begin{example}
In the tableau below we have $N=4$ and $L=6$. The middle tableau is the $\CP_{0}$.
\begin{eqalign}
&\ytableausetup{textmode} \begin{ytableau}
\none &\none & \none &\none&\none   &\none &\none & \none & \\ 
\none &\none & \none &\none&\none   &-3 &-3& -3\\ 
\none &\none &\none &\none&\none& -2&-2&-2 \\ 
\none &\none &\none &\none &\none &-1&1\\
\none &\none &\none &\none & \none &0&2 \\
 \none&\none &\none&1& 1  &1  \\
 \none&\none &\none &2 &2&2 \\
 \none&\none &\none&3 &5 \\
 \none&\none &\none&4 &6 \\
 \none&5& 5  &5  \\
 \none&6 &6&6 \\
 \none&7 &9 \\
 \none&8 &10 \\
  \\
\end{ytableau}
\end{eqalign}    

\end{example}
In this example we observe the presence of the entries $5,6$ in $\CP_{0}$. As we will explain below in the dual-RSK, these originated from \textit{overflowing} from the period $\CP_{1}$. 
\begin{definition}
Suppose for $\CP_{0}$ the partition is $\lambda=(\lambda_{1},...,\lambda_{N})$ , where $\lambda_{i}\leq \lambda_{i-1}$. When
\begin{equation}
\lambda_{1}=L-N+\lambda_{N},    
\end{equation}
then we call this an \textit{overflow}. More generally, if we have $\lambda_{1}=...=\lambda_{k}=L-N+\lambda_{N},   $ then we call this an \textit{overflow of the first k rows}.
\end{definition}
\begin{example}
In the tableau below we have $N=4,L=6$ and $\lambda=(3,3,1,1)$ and so we have overflow of the first and second row
\begin{eqalign}
\lambda_{1}=\lambda_{2}=3=6-4+1=L-N+\lambda_{4}.    
\end{eqalign}
\begin{eqalign}
&\ytableausetup{textmode} \begin{ytableau}
\none &\none & \none &\none&\none   &\none &\none & \none & \\ 
\none &\none & \none &\none&\none   &-3 &-3& -3\\ 
\none &\none &\none &\none&\none& -2&-2&-2 \\ 
\none &\none &\none &\none &\none &-1\\
\none &\none &\none &\none & \none &0 \\
 \none&\none &\none&1& 1  &1  \\
 \none&\none &\none &2 &2&2 \\
 \none&\none &\none&3  \\
 \none&\none &\none&4  \\
 \none&5& 5  &5  \\
 \none&6 &6&6 \\
 \none&7  \\
 \none&8  \\
  \\
\end{ytableau}
\end{eqalign}    

\end{example}

We also assign coordinates to each box.
\begin{definition}(Coordinates)
For a given entry $b$ we will use the notations
\begin{eqalign}
\spara{b}_{N}:=&(b)\text{mod}N\tand \cwpara{b}_{N}:=\branchmat{\frac{1}{N}\para{b-(b)\text{mod}N  }&\tcwhen (b)\text{mod}N\neq 0   \\ \frac{b}{N}-1  &  \tcwhen (b)\text{mod}N=0},
\end{eqalign}
in order to extract the winding multiple eg. if $b=1+kN$, then $\cwpara{b}_{N}=k$ and if $b=N+kN$, then $\cwpara{b}_{N}=N$. Each entry $b$ is the $N$-shift of an entry from the main period
\begin{eqalign}
\spara{b}_{0,N}:=\branchmat{b-\cwpara{b}_{N}N &\tcwhen \cwpara{b}_{N}\neq 0 \\ b & \tcwhen\cwpara{b}_{N}=0  }.    
\end{eqalign}
We assign the coordinates $(i,j,k)$ to a box $b$ when $\cwpara{b}_{N}=k$ and $\spara{b}_{0,N}$ is located on the ith-row and jth column of $\CP_{0}$.     
\end{definition}
We will also need to incorporate general initial data by using skew tableaus.
\begin{definition}(Cylindrical Skew Infinite Young tableau (CISYT))
We fix partitions $\mu,\lambda$. Then we let $\CP_{0}=\lambda/\mu$ and all the other period-tableaus be then shifts of it.
\end{definition}
\begin{example}
In the tableau below we have $N=4$, $L=6$, $\mu:=(2,2,1)$ and $\lambda=(3,2,2,1)$. 
\begin{eqalign}
&\ytableausetup{textmode} \begin{ytableau}
\none &\none & \none &\none&\none   &\none &\none & \none & \\ 
\none &\none & \none &\none&\none  &*(cyan) & *(cyan)  &-3  \\
\none &\none &\none &\none&\none &*(cyan) &*(cyan) \\ 
\none &\none &\none &\none &\none & *(cyan)&-1 \\
\none &\none &\none &\none & \none &0 \\
 \none&\none &\none& *(cyan)& *(cyan)  &1  \\
 \none&\none &\none & *(cyan)& *(cyan)\\
 \none&\none &\none&*(cyan) &3 \\
 \none&\none &\none&4 \\
 \none &*(cyan) &*(cyan)   &5 \\
 \none &*(cyan) &*(cyan)  \\
 \none&*(cyan) & 6\\
 \none&8  \\
  \\
\end{ytableau}
\end{eqalign}    
\end{example}

\subsection{A Cylindrical dual-RSK}
In this section we go over a variation of dual-RSK in the context of CIYT and CISYT.

\begin{definition}\label{def:cylindricalRSKalgorithm}(Cylindrical dual-RSK for CIYT)
We fix $N,L\geq 1$ with $L-N\geq 1$. We start with a 2d-array 
\begin{equation}
\left(\begin{matrix}
a_{1} &a_{2}  &a_{3} &\cdots \\
b_{1}& b_{2} & b_{3} &\cdots 
\end{matrix}\right)    
\end{equation}
where we follow the \textit{lexicographic array construction}  $b_i \leq b_{i+1}$ if $a_i = a_{i+1}$ with $b_{i}\in [N]$ and $a_{i}\in \mathbb{N}$. First suppose we don't have an overflow. Then we follow \textit{column insertion} in $\CP_{0}$ as in the usual dual-RSK 
\begin{itemize}
    \item If every entry in the column is strictly smaller than b, then b is appended to the end of the column.
    \item Otherwise, b is used to replace the uppermost entry in the column which is greater than or equal to b.
\end{itemize}
We repeat this insertion for all the shifted tableaus but with shifted entries i.e. for $\CP_{m}$ we column-insert $b+mN$. Now suppose that we have an overflow of the first $k$ rows.
\begin{itemize}

    \item  If we column-insert an entry $b$ with $1\leq b\leq k$, then one of the boxes $b_{1}$ from $\CP_{0}$ will get displaced into the period-tableau $\CP_{-1}$. In turn, there is a possibility that this box $b_{1}$ pushes yet another box $b_{2}$ into $\CP_{-2}$. This process terminates because we have at most N-rows and so there are at most N boxes to displace. 
    
    \item Once this process is completed for $\CP_{0}$. We update all the period-tableaus to match the new entries of $\CP_{0}$. 
    
\end{itemize}
 As in dual-RSK, we also keep track of the timing of creation of every new box using a tableau $\CQ$. However, in addition we also need to keep track in a tableau $\CW$ the \textit{"winding"} of each insertion path i.e. the number of times the insertion path crossed between periods.
\end{definition}
Here we go over an example.
\begin{example}\label{ex:cylindricalrskex1}
We fix $N=4$ and $L=8$. We start with the following innovation data:
\begin{equation}
 \begin{pmatrix}
i_{1}& i_{2} &\cdots &i_{M} \\ 
j_{1}& j_{2} &\cdots &j_{M} 
\end{pmatrix}=\begin{pmatrix}
1& 1 &1 &2  &2  &3  &3  &3  &3  &4 & 4& 5  & 6 & 7 &7  & 8 & 9 & 9   \\ 
1& 3 &4 &2 &3 &1 & 2 & 3 &  4& 1& 4 & 1 & 1& 1& 2&1 & 1& 2
\end{pmatrix},   
\end{equation}
where $M=18$.
\begin{enumerate}

\item We start with an empty tableau $\CP(0)$ for the particles, an empty tableau $\CQ(0)$ for recording the creation time and $\CW(0)$ for recording the winding:
\begin{eqalign}
&\CP(0)=\ytableausetup{textmode} \begin{ytableau}
\none  &\none &\none &\none &\none &\none &\none &\none &\none & \none & \none& \none & \none& \none&~\\
\none  &\none &\none &\none &\none &\none &\none &\none &\none & ~ & ~& ~ & ~& ~\\
\none  &\none &\none &\none &\none &\none &\none &\none &\none & ~ & \none & \none \\
\none  &\none &\none &\none &\none &\none &\none &\none &\none & ~ \\
\none  &\none &\none &\none &\none &\none &\none &\none &\none & ~ \\
\none  &\none &\none &\none &\none&~ & ~ & ~& ~ & ~  \\ 
\none  &\none &\none &\none &\none  &~ & \none & \none \\
\none  &\none  &\none &\none &\none &~ &\none  &\none \\
\none  &\none  &\none &\none &\none &~ &\none  &\none \\
\none  &~&  ~ & ~& ~ & ~\\
\none  &~ & \none & \none \\ 
\none  &~ & \none & \none\\
\none  &~ & \none & \none\\
~\end{ytableau}\\
&,\CQ(0)=\ytableausetup{textmode} \begin{ytableau} ~& \none &\none &\none &\none \\ ~ & \none & \none \\ ~ & \none & \none \\~ & \none & \none\end{ytableau}  \tand \CW(0)=\ytableausetup{textmode} \begin{ytableau} ~& \none &\none &\none &\none \\ ~ & \none & \none \\ ~ & \none & \none \\~ & \none & \none\end{ytableau}.
\end{eqalign}

\item We input $\binom{1}{1}$,$\binom{1}{3}$,$\binom{1}{4}$ to get
\begin{eqalign}
&\CP(1)=\ytableausetup{textmode} \begin{ytableau}
\none  &\none &\none &\none &\none &\none &\none &\none &\none & \none & \none& \none & \none& \none&~\\
\none &\none &\none &\none &\none &\none &\none &\none &\none & -3 &  ~ & ~& ~ & ~ \\ 
\none &\none &\none &\none &\none &\none &\none &\none &\none & -1 & \none & \none \\ 
\none &\none &\none &\none &\none &\none &\none &\none &\none & 0 \\
\none &\none &\none &\none &\none &\none &\none &\none &\none & ~ \\
\none &\none &\none &\none &\none&1 & ~ & ~& ~ & ~ \\
\none &\none &\none &\none &\none  &3 & \none & \none \\
\none &\none  &\none &\none &\none &4 &\none  &\none \\
\none &\none  &\none &\none &\none &~ &\none  &\none \\
\none &5& ~ & ~& ~ & ~ \\
\none &7 & \none & \none \\
\none &8& \none & \none \\
\none &~& \none & \none \\
~
\end{ytableau}\\
& \CQ(1)=\ytableausetup{textmode} \begin{ytableau} 1& \none &\none &\none &\none \\ 1 & \none & \none \\ 1 & \none & \none\\~ & \none & \none \end{ytableau}. \tand  \CW(1)=\ytableausetup{textmode} \begin{ytableau} 0& \none &\none &\none &\none \\ 0 & \none & \none \\0& \none & \none\\~ & \none & \none \end{ytableau}.   
\end{eqalign}

\item We input $\binom{2}{2},\binom{2}{3}$ to get
\begin{eqalign}\label{eq:step2tableau}
&\CP(2)=\ytableausetup{textmode} \begin{ytableau}
\none  &\none &\none &\none &\none &\none &\none &\none &\none & \none & \none& \none & \none& \none&~\\
\none &\none &\none &\none &\none &\none &\none &\none &\none & -3 &  -1 & ~& ~ & ~ \\ 
\none &\none &\none &\none &\none &\none &\none &\none &\none & -2 & 0 & \none \\ 
\none &\none &\none &\none &\none &\none &\none &\none &\none & -1 \\
\none &\none &\none &\none &\none &\none &\none &\none &\none & ~ \\
\none &\none &\none &\none &\none&1 & 3 & ~& ~ & ~ \\
\none &\none &\none &\none &\none  &2 & 4 & \none \\
\none &\none  &\none &\none &\none &3 &\none  &\none \\
\none &\none  &\none &\none &\none &~ &\none  &\none \\
\none &5& 7 & ~& ~ & ~ \\
\none &6 & 8 & \none \\
\none &7& \none & \none \\
\none &~& \none & \none \\
~
\end{ytableau}\\
&\CQ(2)=\ytableausetup{textmode} \begin{ytableau} 1& 2 &\none &\none &\none \\ 1 & 2 & \none \\ 1& \none & \none\\~ & \none & \none \end{ytableau}, \tand \CW(2)=\ytableausetup{textmode} \begin{ytableau} 0& 0 &\none &\none &\none \\ 0 & 0 & \none \\ 0& \none & \none\\~ & \none & \none \end{ytableau}.    
\end{eqalign}

\item Next we enter the times 3 and 4 boxes $ \bigl(\begin{smallmatrix}
3  &3  &3  &3  &4 &4 \\ 
1 & 2 & 3 &  4& 1& 4 
\end{smallmatrix}\bigr)$ to obtain 
\begin{eqalign}
&\CP(4)=\ytableausetup{textmode} \begin{ytableau}
\none  &\none &\none &\none &\none &\none &\none &\none &\none & \none & \none& \none & \none& \none&~\\
\none &\none &\none &\none &\none &\none &\none &\none &\none & -3 &  -3 & -3& -1 & ~ \\ 
\none &\none &\none &\none &\none &\none &\none &\none &\none & -2 & -2 & 0 \\ 
\none &\none &\none &\none &\none &\none &\none &\none &\none & -1 & -1\\
\none &\none &\none &\none &\none &\none &\none &\none &\none & 0 & 0 \\
\none &\none &\none &\none &\none&1 & 1 & 1& 3 & ~ \\
\none &\none &\none& \none &\none  &2 & 2 &4 \\
\none  &\none &\none& \none &\none &3 &3  &\none \\
\none  &\none &\none&\none & \none &4 &4  &\none \\
\none &5& 5 & 5& 7 & ~ \\
\none &6 & 6 & 8 \\
\none &7& 7 & \none \\
\none &8& 8& \none \\
~
\end{ytableau} ,\\
& \CQ(4)=\ytableausetup{textmode} \begin{ytableau} 1& 2 &3 &4 &\none \\ 1 & 2 & 3 \\ 1& 3 & \none\\3 & 4 & \none \end{ytableau} \tand \CW(4)=\ytableausetup{textmode} \begin{ytableau} 0& 0 &0 &0 &\none \\ 0 & 0 & 0\\ 0& 0 & \none\\0 & 0& \none \end{ytableau}.
\end{eqalign}
Therefore, we see that by having the first two columns filled, we allowed the first row of the previous period to expand further. 

\item Next we want to check what happens when the first rows get overfilled by inserting the data  $\bigl(\begin{smallmatrix} 5  & 6 & 7& 7  \\ 
1 & 1& 1&2
\end{smallmatrix}\bigr)$
\begin{eqalign}\label{eq:steptableauoverflow}
&\CP(7)=\ytableausetup{textmode} \begin{ytableau}
\none  &\none &\none &\none &\none &\none &\none &\none &\none & \none & \none& \none & \none& \none&\none&~\\
\none &\none &\none &\none &\none &\none &\none &\none &\none & -3 &  -3 & -3& -3 & -3 &-3\\ 
\none &\none &\none &\none &\none &\none &\none &\none &\none & -2 & -2 &-2& 0 \\ 
\none &\none &\none &\none &\none &\none &\none &\none &\none & -1 & -1 & 3\\
\none &\none &\none &\none &\none &\none &\none &\none &\none & 0 & 0 \\
\none &\none &\none &\none &\none&1 & 1 & 1& 1 &1  &1  \\
\none &\none &\none &\none &\none  &2 & 2 &2&4 \\
\none &\none  &\none &\none &\none &3 &3  &7 \\
\none &\none  &\none &\none &\none &4 &4  &\none \\
\none &5& 5 & 5& 5 & 5& 5 \\
\none &6 & 6 &6& 8 \\
\none &7& 7 & 11\\
\none &8& 8& \none \\
~
\end{ytableau}\\
&,\CQ(7)= \ytableausetup{textmode} \begin{ytableau} 1& 2 &3 &4 &5& 6\\ 1 & 2 & 3& 7 \\ 1& 3 &7\\3 & 4 & \none \end{ytableau} \tand\CW(7)= \ytableausetup{textmode} \begin{ytableau}0& 0 &0 &0 &0& 0\\ 0 & 0 & 0& 0 \\ 0& 0 &1\\0 & 0 & \none \end{ytableau}.
\end{eqalign}
Therefore, boxes 3,7,9 were pushed into the next period by keep adding particle 1 boxes. We also note here that the last insertion $\binom{7}{2}$ is on the north of the penultimate insertion $\binom{7}{1}$ whereas in the previous above "internal"-cases, the last insertion was always the southernmost. We finally see here that we added the label "$1$" in the new box for $\CW$ because inserting $\binom{7}{1}$ led to an insertion path that crossed from $\CP_1$ to $\CP_0$.

\item We can also break the strictly increasing rule for $\CQ$'s rows. We insert  $\bigl(\begin{smallmatrix}
 8 & 9 & 9   \\ 
1 & 1& 2
\end{smallmatrix}\bigr)$ to get
\begin{eqalign}
&\CP(9)=\ytableausetup{textmode} \begin{ytableau}
\none  &\none &\none &\none &\none &\none &\none &\none &\none & \none & \none& \none & \none& \none&\none&~\\
\none &\none &\none &\none &\none &\none &\none &\none &\none & -3 &  -3 & -3& -3 & -3 &-3\\ 
\none &\none &\none &\none &\none &\none &\none &\none &\none & -2 & -2 & -2&-2&0&3 \\ 
\none &\none &\none &\none &\none &\none &\none &\none &\none & -1 & -1 & 1&1\\
\none &\none &\none &\none &\none &\none &\none &\none &\none & 0 & 0 \\
\none &\none &\none &\none &\none&1 & 1 & 1& 1 &1  &1  \\
\none &\none &\none &\none &\none  &2 & 2&2 &2&4&7 \\
\none &\none  &\none &\none &\none &3 &3  &5&5 \\
\none &\none  &\none &\none &\none &4 &4  &\none \\
\none &5& 5 & 5& 5 & 5& 5 \\
\none &6 & 6 &6&6& 8&11 \\
\none &7& 7 & 9&9\\
\none &8& 8& \none \\
~
\end{ytableau}\\
&\CQ(9)=\ytableausetup{textmode} \begin{ytableau} 1& 2 &3 &4 &5& 6\\ 1 & 2 & 3 &7&9&9\\ 1& 3 &7&8\\3 & 4 & \none \end{ytableau} \tand\CW(9)= \ytableausetup{textmode} \begin{ytableau}0& 0 &0 &0 &0& 0\\ 0 & 0 & 0& 0&1&0 \\ 0& 0 &1&1\\0 & 0 & \none \end{ytableau}.
\end{eqalign}

\end{enumerate}
Therefore, we get two $9$'s in the same row. This means that the rows of $\CQ$ are only weakly increasing, not \textit{strictly} as in the case of Dual RSK. And as before, we had to add two boxes in $\CW$ with label "$1$" from inserting $\bigl(\begin{smallmatrix}
 8 & 9  \\ 
1 & 1
\end{smallmatrix}\bigr)$ that crossed from $\CP_1$ to $\CP_0$, and we also added a box in $\CW$ with "$0$"  from inserting $\binom{9}{2}$.

\end{example}

Here is the modification for CISYT.
\begin{definition}\label{def:cylindricalRSKalgorithmCISYT}(Cylindrical dual-RSK for CISYT)
Here we fix partitions $\mu=(\mu_{1},...,\mu_{N}),\lambda=(\lambda_{1},...,\lambda_{N})$. To return to the previous setting, we simply fill in the boxes corresponding to $\mu$ in a standard way: we input entry $k$ on the kth row of the augmented tableau $\CP_{0}\cup \mu$ up to $\mu_{k}$ (if $\mu_{k}=0$, we don't add anything). Then we do the column-insertion as per usual but now starting from the augmented tableau $\CP_{0}\cup \mu$.
\end{definition}

\subsection{Periodic TASEP and CISYT}
\pparagraph{$\ZTASEP$ and  semi-standard Young
tableaux (SSYT)}
We quickly recall the case B of \textit{Bernoulli jumps with blocking} from \cite{DW08}. Particles are labelled from right to left, so $Y_N (n) \leq Y_{N-1}(n) \leq\cdots\leq Y_1(n)$. Between time $n-1$ and $n$ each particle attempts to move one step to the right, but it is constrained not to overtake the particle to its right. Particle $i$ moves with probability $p_i$. The particles are now updated from right to left, so it is the
updated position of the particle to the right that acts as a block.\\
The evolution is generated from a family $(\xi(k,n): k \in \set{1, 2,...,N},n \in N)$ of independent Bernoulli random variables satisfying $P(\xi(k,n) = +1) = 1 - P(\xi(k,n) = 0) = p_k$ , via the recursions $Y_1(n) = Y_1(n - 1) + \xi(1,n)$, and for $k = 2, 3,...,N$,
\begin{eqalign}
Y_k(n) = \minp{Y_k(n - 1) + \xi(k,n),Y_{k-1}(n)}
\end{eqalign}
In particular, $Y$ is a Markov chain on $W^{N}:=\set{z\in \zz^{N}: z_{N}\leq z_{N-1}\leq  \cdots \leq z_{1}}$. On shifting
the ith particle $i$ positions to the left, $Y$ corresponds to the discrete-time totally asymmetric simple exclusion process (TASEP) with sequential updating.\\
Finally, in the context of SSYT $(\CP(n),\CQ(n))_{n\geq 0}$, the process $\set{Y(n)}_{n\geq 0}$ was equal to the $\set{\text{ledge }(\CP(n))}_{n\geq 0}$, which is the the $N$-vector for which element $i$ is the number of $i$'s in row $i$ of $\CP(n)$. In words, whenever we come cross in the innovation matrix an entry of the form $\binom{n}{k}$, it means that the clock rung for the $k$th particle to move to the right and that we will column-insert the entry $k$ in the tableau. If the $(k-1)$th particle is blocking it, it means that in the tableau we are missing a $k-1$ entry and so when we insert $k$, it falls into a row $r$ that is strictly smaller than $k$.
\pparagraph{$\PTASEP$} 
Here we again generate the same family $(\xi(k,n): k \in \set{1, 2,...,N},n \in N)$ of independent Bernoulli random variables satisfying $P(\xi(k,n) = +1) = 1 - P(\xi(k,n) = 0) = p_k$. However, we change the recursion for the first particle to prevent it from overtaking the shifted particle $Y_{N}+L$
\begin{eqalign}
Y_1(n) =& \minp{Y_{1}(n - 1) + \xi(1,n),Y_{N}(n-1)+L}\\
Y_k(n) =& \minp{Y_k(n - 1) + \xi(k,n),Y_{k-1}(n)},k=2,...,N
\end{eqalign}
\pparagraph{$\PTASEP$ and CISYT}Here too we have that the $\set{\text{ledge }(\CP_{0}(n))}_{n\geq 0}$ is equal to $\PTASEP$ process $\set{Y(n)}_{n\geq 0}$. The interior exclusion dynamics are captured as for usual $\ZTASEP$. The boundary exclusion of the first particle $Y_{1}$ getting blocked by the shifted $Y_{N}+L$ is captured by the \textit{overflown} situation we discussed above. In particular, the first column of the tableau $\CP_{0}$ is shifted back by $L-N$ behind the first column of $\CP_{-1}$. That means that we leave a room of $L-N$-sites for the particle $Y_{1}$ to move unobstructed. However, once the number of $1$'s in the first row reaches the $L-N$ column, then the next $1$ will either fall into $\CP_{-1}$ or it will continue on the first row if the $(L-N+1)$th column is filled with entries $1,...,N$. This because for the $Y_{1}$ particle to move, we need the $Y_{N}$ particle to have moved at least once i.e. the first column is filled. Whereas if $Y_{N}$ didn't move, then the inserted entry $1$ will fall into $\CP_{-1}$ and thus not add to the ledge count. 


\newpage 
\section{Bijection of cylindrical Dual RSK}
In this section we will prove a bijection between arrays
\begin{equation}\label{eq:innovationdata2}
 \begin{pmatrix}
i_{1}& i_{2} &\cdots &i_{M} \\ 
j_{1}& j_{2} &\cdots &j_{M} 
\end{pmatrix}.      
\end{equation}
and triples $(\CP(n),\CQ(n),\CW(n))_{n=1}^{m}$. We will use ideas from the bijection for usual RSK/Dual RSK \parencite[Theorem 7.11.5,Theorem 7.14.1]{stanley2011enumerative}. It is similar as in the Dual RSK case with additions that we will emphasize coming the infinite structure of the above SSYT. One major difference is that we need to keep track of the winding numbers of each insertion path. The reason we also need the winding tableau $\CW$ is because there is an ambiguity on whether a newly created box originated from an insertion path that fully stayed within $\CP_{0}$ or crossed between periods e.g. in \cref{eq:steptableauoverflow} we don't know if the entries at time "7" were interior, cylindrical or some combination.
\subsection{Some properties of the insertion paths in CISYT}
\noindent Here we collect some properties of the cylindrical insertion paths that we will denote as $I_{C}$. We will need a corresponding notion of an insertion path $I(b\rightarrow \CP)$, namely the set of all the $(x,y)$-coordinates where changes occurred. For example, in \cref{eq:step2tableau} inserting $\binom{2}{2}$ has the insertion path
\begin{equation}
I_{C}(2\rightarrow \CP(1))=\set{(2,1),(1,2) },    
\end{equation}
because it displaced the box with entry $3$ to the second column. Whereas in \cref{eq:steptableauoverflow} inserting $\binom{7}{1}$ caused a cylindrical displacement 
 has the insertion path
\begin{equation}
I_{C}(1\rightarrow \CP(6))=\set{(1,6),(3,3) },    
\end{equation}
because it displaced the box with entry $3$ which fell into the next period and so the corresponding entry $7$ fell into the third column at $(3,3)$. 
\begin{definition}(Cylindrical insertion path $I_{C}$)
\begin{itemize}
    \item  If the insertion path didn't move past the first row, then we call $I_{C}$ is an \textit{interior} insertion path.

    \item  But if moved past the first row, we call $I_{C}$ a \textit{cylindrical} insertion path $A$ with winding $\omega$ i.e. the number of times it crossed through periods. 

\item A cylindrical insertion path $A$ consists of possibly many \textit{iterations} $A^{k},k=1,...,\omega$ of insertions paths moving past the first row each time. The \textit{last iteration} $A^{\omega}$ can be viewed as an interior insertion of some box from the previous period $\CW=1$. The path in \cref{eq:steptableauoverflow} has $\omega=1$ since it crossed the first row only once. 

    \item By \textit{last box} or \textit{new box} of an insertion path we mean the last entry in the the insertion path which denotes the location of the new box that we added in the tableau eg. in $I_{C}(1\rightarrow \CP(6))$, the $(3,3)$ is the location of the new box.

    \item We denote by $\CS A^{k}$ the periodic copy of $A^{k}\in \CP^{0}$ in the tableau $\CP^{1}$ i.e. we shift every location by the vector $(-(L-N),N)$.
    
\end{itemize}

\end{definition}
Cylindrical iterations only happen when we have an overflow i.e. the first row has grown to cover the columns underneath. We call this part of the tableau the \textit{wall}.
\begin{definition}
The \textit{wall of $\CP^{n}$} are all the columns that are below the first row of the next period $\CP^{n-1}$. The \textit{edge of the wall of $\CP^{n}$} is the very last column. For example, at step $t=6$ in \cref{ex:cylindricalrskex1}, the first and second column consisting of $1,2,3,4$ are the wall for $\CP^{0}$ and the second column is the edge of the wall. But for example in $t=4$, the wall is empty because the first row didn't expand long enough to include the two columns 1,2,3,4. In the case of $n=0$, we will just call it the \textit{wall} and \textit{edge} of wall.
\end{definition}



The following lemmas describe the ordering of cylindrical path insertions for various windings.
\begin{lemma}\label{it:(1)orderingofcylindricalinsertionpaths}
 All the cylindrical insertions happened before the internal insertions.    
\end{lemma}
\begin{proof}
We start with entries $w_{1}<w_{2}$. Suppose by contradiction that inserting entry $w_{1}$ resulted in an interior-insertion whereas inserting entry $w_{2}$ relulted in a cylindrical one. In the first iteration the cylindrical insertion path needs to move past the first row. From \cref{lem:insertionpathproperties}, we have that the insertion path of $w_{1}$ is above  the first iteration-path of $w_{2}$. So the last two sentences contradict each other since the path of $w_{1}$ doesn't move past the first row, whereas the insertion-path of $w_{2}$ needs to move past it.
\end{proof}
We have the following more general statement.
\begin{lemma}\label{it:(3)orderingofcylindricalinsertionpaths}
Suppose that inserting entries $a_{1}<a_{2}$,with $a_{1},a_{2}\in[N]$, results to insertions $A,B$ with windings $\omega_{1},\omega_{2}\geq 0$. Then we have that the path iterations are ordered
\begin{eqalign}
    A^{k+1}\preceq B^{k}\prec A^{k}, k\in [1, \omega_{2}],
\end{eqalign}
the windings satisfy
\begin{eqalign}
\omega_{1}\geq \omega_{2},    
\end{eqalign}
and the new box $r_{1}$ created by $A$ is weakly above and right of the new box $r_{2}$ created by $B$. If $\omega_{1}=\omega_{2}$, we further get that $r_{1}$ is northeast of $r_{2}$.
\end{lemma}
\begin{proof}
Since $a_{1}<a_{2}$, we mean that the entire insertion path $A$ already happened before the insertion path $B$ occurred. We assume that $\omega_{1}\geq 1$, otherwise the proposition follows from \cref{lem:insertionpathproperties}. We start with comparing the first iterations for $k=1$ and the rest follow by bootstrapping i.e. we will show that
\begin{eqalign}
    A^{2}\preceq B^{1}\prec A^{1}.
\end{eqalign}
Since we have $a_{1}<a_{2}$, the \cref{lem:insertionpathproperties} implies that $ B^{1}\prec A^{1}$. 
\begin{enumerate}
    \item Since $\omega_{1}\geq 1$, it means we have an \textit{overflow} and so there is an edge-wall that $A^{1}$ will reach to bump some entry $g_{1}(a_{1})$. This means that the shifted entry $g_{1}(a_{1})+N$ will get inserted in the first column of $\CP_{0}$. This will create the iteration $A^{2}$. 

    \item If $\omega_{2}=0$, then there is no second iteration $B^{2}$ and so we are done. So suppose that $\omega_{2}\geq 1$. Then similarly, we will insert the entry $g_{2}(a_{2})+N$ in the first column of $\CP_{0}$.
    
    \item By the ordering $ B^{1}\prec A^{1}$, we get
    \begin{eqalign}
       g_{1}(a_{1})<g_{2}(a_{2})\Rightarrow g_{1}(a_{1})+N<g_{2}(a_{2})+N.
    \end{eqalign}
    
    \item Also,    we see that $g_{1}(a_{1})+N\geq 1+N>a_{2}$. So here we are in the part (3) of \cref{lem:insertionpathproperties}, where $x:=g_{1}(a_{1})>a_{2}=:x'$ and the insertion path $A^{2}$ happened before $B^{1}$. This gives
    \begin{eqalign}
       A^{2}\preceq B^{1}.  
    \end{eqalign}
    
\end{enumerate}
Next we show $\omega_{1}\geq \omega_{2}$. Suppose that $\omega_{2}\geq 1+\omega_{1}$. The iteration $A^{\omega_{1}}$ is the final one and so it created a new box $r_{1}$. The $B^{\omega_{1}}$ is the penultimate iteration and so that means it reached the edge-wall to bump some entry $r_{2}$. However, we have that $B^{\omega_{1}}\prec A^{\omega_{1}}$ which means that $r_{1}$ is created in a smaller row compared to the row of $r_{2}$. This is a contradiction because to have an overflow at row $r_{2}$, we need to have an overflow at $r_{1}$ too. \\
Therefore, we also get that the box with $r_{1}$ is weakly above the box with $r_{2}$. However, if we further know that $\omega_{1}=\omega_{2}$, then we can use that $B^{\omega_{1}}\prec A^{\omega_{1}}$ to get that the box $r_{1}$ is strictly above/right of the box $r_{2}$.    
\end{proof}
This lemma has a converse statement.
\begin{lemma}\label{it:(4)orderingofcylindricalinsertionpaths}
Suppose that inserting distinct entries $a_{1}\neq a_{2}\in[N]$, results to insertions $A,B$ with windings $\omega_{1},\omega_{2}\geq 0$ satisfying $\omega_{1}\geq \omega_{2}$. Also suppose that insertion $A$ creates a new box $r_{1}$ that is northeast of the new box $r_{2}$ created by $B$. Then we have that the original entries satisfy
\begin{eqalign}
  a_{1}<a_{2},  
\end{eqalign}
and so the path iterations are ordered
\begin{eqalign}
    A^{k+1}\preceq B^{k}\prec A^{k}, k\in [1, \omega_{2}].
\end{eqalign}

\end{lemma}
\begin{proof}
If $\omega_{1}>\omega_{2}$, then by the contrapositive of \cref{it:(3)orderingofcylindricalinsertionpaths}, we get $a_{1}\leq a_{2}$ and since they are distinct we get $a_{1}<a_{2}$. Now suppose that $\omega_{1}=\omega=\omega_{2}$. The last segments of $A^{\omega}$ and $B^{\omega}$ are internal insertions since they have matching winding. If $a_{1}>a_{2}$, then by \cref{it:(3)orderingofcylindricalinsertionpaths} the last entry $r_{1}$ is  below/left that of $r_{2}$, which contradicts our assumption.    
\end{proof}

\subsection{Bijection}
Here we prove the bijection.

\begin{theorem}\label{thm:bijectioncylindricaldualRSK}
We have a bijection between $(\CP(n),\CQ(n),\CW(n))$ and the innovation data array \cref{eq:innovationdata2}.
\end{theorem}
\begin{proof}
\pparagraph{Injectivity} Here we will show that the algorithm is invertible onto its image and hence injective, meaning that we start with triple $(\CP(n),\CQ(n),\CW(n))$ and obtain the unique pair $\binom{i_{m}}{j_{m}}$ that led to it.  
\begin{enumerate}
    \item First we single out all the boxes in $\CQ(n)$ that contain $n=i_{m}$.

    \item We look at the corresponding boxes in $\CW(n)$. Their values are some sequence $\{\omega_{1}..,\omega_{\ell}\}$ with $0\leq \omega_{1}<\cdots<\omega_{\ell}$. 
    
    \item By \cref{it:(3)orderingofcylindricalinsertionpaths}, we focus on the boxes with $\omega_{1}$ because these correspond to the group of insertions $A_{k_1},...,A_{k_r}$ that were made last. 

    \item From particularly \cref{it:(4)orderingofcylindricalinsertionpaths}, we see that the insertion paths $A_{k_1},...,A_{k_r}$ follow the usual Dual-RSK. Therefore, we single out the rightmost boxes and among them we pick the southernmost box. This is the last box that got inserted. From here we proceed as in the proofs of \parencite[Theorem 7.11.5,Theorem 7.14.1]{stanley2011enumerative} to extract $\binom{i_{m}}{j_{m}}$.
\end{enumerate}

\pparagraph{Surjectivity}
Here we repeat the above inversion but for two consecutive boxes with $n=i_{k}=i_{k+1}$ and show that $j_{k}<j_{k+1}$. 
\begin{enumerate}
    \item We single out all the boxes in $\CQ(n)$ that contain $n$. These have corresponding windings $\{\omega_{1}..,\omega_{\ell}\}$ with $0\leq \omega_{1}<\cdots<\omega_{\ell}$.

    \item First suppose that we have at least two entries in $\CW(n)$ with the smallest $\omega_{1}$. Then we go to the corresponding entries in $\CP(n)$ and find the two entries $r_{1},r_{2}$ that are the most northerneast. By \cref{it:(3)orderingofcylindricalinsertionpaths}, since they come from two insertion paths $A,B$ with the same $\omega_{1}$, we know that ,say, $r_{1}$ will be northeast of $r_{2}$. And so we can invert this path to get $j_{k}< j_{k+1}$. 

    \item Now suppose that $\omega_{1}$ has a single box $r_{1}$ with a corresponding insertion path $A$.  From $\omega_{2}$, we pick the most northeast box $r_{2}$ with a corresponding insertion path $B$. Here are in the setting of \cref{it:(4)orderingofcylindricalinsertionpaths} and so we again get $a_{1}<a_{2}$.
    



\end{enumerate}

\end{proof}
\begin{remark}
There is a possibility that the tableau $\CW(n)$ can be extracted from $(\CP(n),\CQ(n))$. We can try to extract windings by doing a contradiction argument based on insertion paths. We start from the southern and most rightmost box, if its insertion path goes over to previous period, then we stop and go to the next box. We keep going till we identify any zero-winding boxes. Then we repeat for all the omega=1 boxes and so on and so forth. At the end we fill the W-tableau of time n with all the various windings. Once we have that information, we proceed with the proof above. 
\end{remark}


\section{Cylindrical Tableaus, Gelfand-Tsetlin patterns and non-intersecting paths }
Here we go over the analogous  correspondence for cylindrical tableaus as in \parencite[section 4]{DW08}. We embed $T^{N,\wedge}_{N}$ in a space parametrized by $\bf{x}=(x^{1},...,x^{N},x^{N+1},...)$ with $x^{k}= (x_{1}^{k},...,x_{\min(k,N)}^{k})\in \mb{Z}^{k}, k\geq 1$ satisfying the following constraints: for $k=2,...,N$ 
\begin{eqalign}
&x_{k}^{k}< x^{k-1}_{k-1}\leq  x^{k-1}_{k-2}\leq \cdots \leq x_{2}^{k}\leq x^{k-1}_{1}< x_{1}^{k},   \\
\tand &x^{N}_{1}\leq L-N+x_{N}^{N},
\end{eqalign}
and for $k=N+1,...$ 
\begin{eqalign}
&x_{N}^{k}\leq x^{N-1}_{k}\leq...\leq x_{1}^{k},\\
&x^{N-1}_{k}\leq x^{N-1}_{k-1}.
\end{eqalign}
Pictorially we have an infinite wedge array
\begin{eqalign}
 \bf{x}=\quad\quad&\begin{matrix}
 \vdots & &\vdots &&&&&\vdots&\vdots\\ 
 x_{N}^{N+2}&\leq & x_{N-1}^{N+2}&\leq &&\cdots&&\leq x_{2}^{N+2}&\leq x_{1}^{N+2}\\
  &  \rotatebox[origin=c]{320}{$\leq $}&    && &&&\rotatebox[origin=c]{320}{$\leq $}&\rotatebox[origin=c]{90}{$\leq $} \\ 
 x_{N}^{N+1}&\leq & x_{N-1}^{N+1}&\leq &&\cdots&&\leq x_{2}^{N+1}&\leq x_{1}^{N+1}\\ 
 &  \rotatebox[origin=c]{320}{$\leq $}&    && &&&\rotatebox[origin=c]{320}{$\leq $}&\rotatebox[origin=c]{90}{$\leq $} \\ 
 x_{N}^{N}&\leq & x_{N-1}^{N}&\leq &&\cdots&&\leq x_{2}^{N}&\leq x_{1}^{N}\\ 
  \rotatebox[origin=c]{320}{$<$}&  & \rotatebox[origin=c]{320}{$<$}   && &&&\rotatebox[origin=c]{320}{$<$}&\rotatebox[origin=c]{60}{$<$} \\ 
  &x_{N-1}^{N-1}&    \leq & x_{N-2}^{N-1}  &&\cdots&&\leq & x_{1}^{N-1} \\
&  \ddots &  &   && &&\udots& \\ 
 &&  \rotatebox[origin=c]{320}{$<$}   &   &&& \rotatebox[origin=c]{60}{$<$}& & \\
  &&    &   x_{2}^{2} &\leq&x_{1}^{2}& & & \\
 &&    &  \rotatebox[origin=c]{320}{$<$}&&\rotatebox[origin=c]{60}{$<$}& &  \\
  &  &  &  &x_{1}^{1}&&& & \\
\end{matrix},    
\end{eqalign}
and $x^{N}_{1}\leq L-N+x_{N}^{N}$ means that the shifted array $\bf{x}+L-N$ is on the right of $\bf{x}$. The correspondence with a tableau $\CT\in T^{N,\wedge}_{N}$ is that $x^{j}_{i}$ ,for $i\in [N],j\geq 1$, mark the position
of the last entry in row $i$ in $\CT$ whose label does not exceed $j$. So to get a specific entry we take differences. Here we allow $j$ to be arbitrarily large to accommodate for entries with $r\geq N+1$. For example, the number $2$ in the first row shows up from $x_{1}^{1}+1$ to $x_{1}^{2}$. We write $K^{N}$ for the set of all $\bf{x}$, called \textit{cylindrical Gelfand-Tsetlin (cGT) patterns}. \\
However, because we only do finitely many iterations of cylindrical-RSK, there is always some large enough number $M_{\CT}\geq N$ such that the array repeats after i.e.
\begin{eqalign}
x_{k}^{M_{\CT}}=x_{k}^{M_{\CT}+1}=x_{k}^{M_{\CT}+2}=\cdots , \tfor k=1,...,N.    
\end{eqalign}
We set 
\begin{itemize}
    
\item $sh(\bf{x}):=(x_{1}^{M_{\CT}},x_{2}^{M_{\CT}},...,x_{N}^{M_{\CT}})$,
    
\item $ledge(\bf{x}):=(x_{1}^{1},x_{2}^{2},...,x_{N}^{N})$,

\item and  $redge(\bf{x}):=(x_{1}^{1},x_{1}^{2}...,x_{1}^{N})$. 
    
\end{itemize}
We have a bijection of cGT-patterns with cylindrical non-intersecting paths. Each array $\bf{x}\in K^{N}$ with $z:=sh(\bf{x})$ and $y:=ledge(\bf{x})$ 
can be encoded by an (N-1)-tuple of non-intersecting paths $(P_{1},...,P_{N-1})$ along the edges of $\mb{Z}^{2}$ and travelling upright. The path $P_{k}$ runs from $(y_{k}-k,k+1)$ to $(z_{k}-k,M_{\CT})$ and it contains the horizontal edges from $(x_{k}^{r-1}-k,r)$ to $(x_{k}^{r}-k,r)$ for $r=k+1,...,M_{\CT}$. For the periodic shifts we consider the shift
\begin{eqalign}
\mathcal{S} (a,b):=(-(L-N)+a,b+M_{\CT})\tand \mathcal{S}^{-1} (a,b):=(L-N+a,b-M_{\CT}).    
\end{eqalign}

\begin{example}
For $N=4,L=7$ we have the following correspondence of the cylindrical tableau
\begin{eqalign}
\CT=&\ytableausetup{textmode} \begin{ytableau}
\none&\none&\none&\none &\none & \none &\none&\none   &\none &\none& \none& ~\\ 
\none&\none&\none &\none & \none &\none&\none   &-3 &-2& -1& -1\\ 
\none&\none&\none &\none &\none &\none&\none& -2&-1&0&0 \\ 
\none&\none&\none &\none &\none &\none &\none &-1&0\\
\none&\none&\none &\none &\none &\none & \none &0&3 \\
 \none&\none&\none &\none&1& 2  &3&3  \\
 \none&\none&\none &\none &2 &3&4&4 \\
 \none&\none&\none &\none&3 &4 \\
 \none&\none&\none &\none&4 &7 \\
 \none&5& 6  &7&7  \\
 \none&6 &7&8&8 \\
 \none&7 &8 \\
 \none&8 &11 \\
 ~
\end{ytableau}\\
\end{eqalign}
with the infinite array
\begin{eqalign}
\bf{x}_{\CT}=\quad\quad &\begin{matrix}&&\vdots&&\\x_{4}^{8}&x_{3}^{8}&x_{2}^{8}&x_{1}^{8}\\x_{4}^{7}&x_{3}^{7}&x_{2}^{7}&x_{1}^{7}\\x_{4}^{6}&x_{3}^{6}&x_{2}^{6}&x_{1}^{6}\\x_{4}^{5}&x_{3}^{5}&x_{2}^{5}&x_{1}^{5}\\x_{4}^{4}&x_{3}^{4}&x_{2}^{4}&x_{1}^{4}\\x_{3}^{3}&x_{2}^{3}&x_{1}^{3}\\x_{2}^{2}&x_{1}^{2}\\x_{1}^{1}
\end{matrix}=\begin{matrix}&&\vdots&&\\2&2&4&4\\2&2&4&4\\1&2&4&4\\1&2&4&4\\1&2&4&4\\1&2&4\\1&2\\1
\end{matrix}
\end{eqalign}
We observe here that $M_{\CT}=7$ since the rows repeat after the $7$th row. The corresponding non-intersecting paths (for $y=(1,1,1,1)$ and $z=(4,4,2,2)$) are the following
\begin{figure}[H]
\centering
\includegraphics[scale=.7]{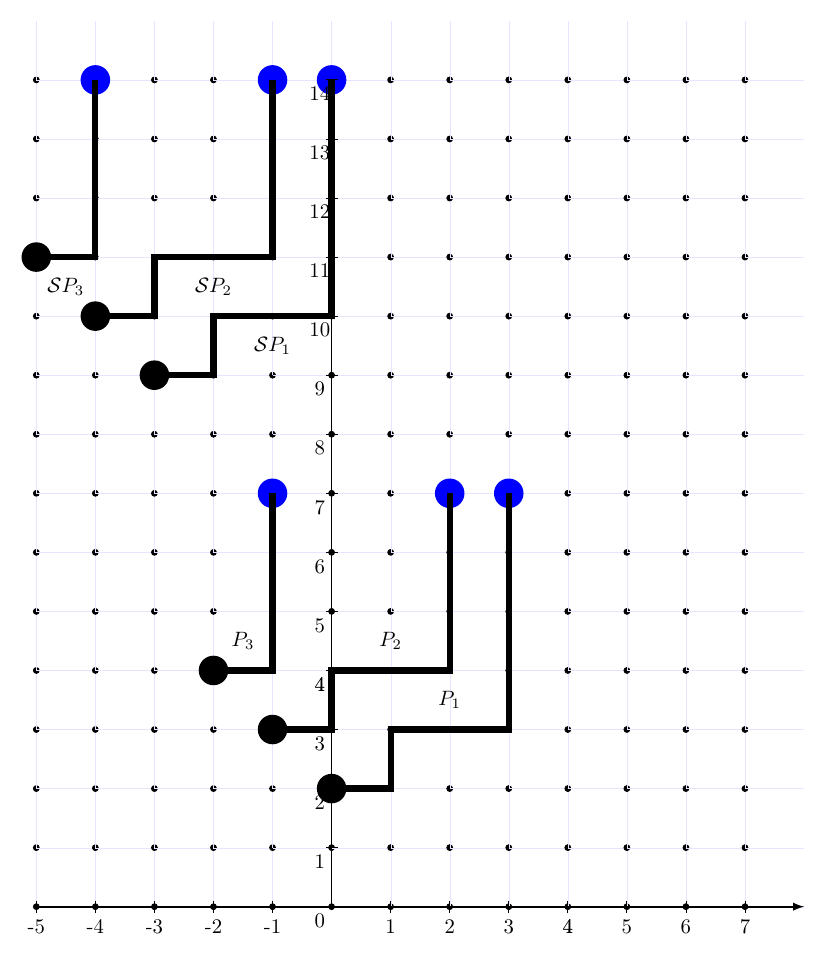}
\end{figure}

\end{example}
\newpage \section{Inversion in \cite{DW08} using flagged Schur and Posets}
In \cite[proposition 3]{DW08}, they construct an inverse operator $\Pi(y,z)$ for $K_{p}$ as a determinant of elementary symmetric polynomials and thus prove inversion using the Cauchy-Binet identity. The issue in the cylindrical case ,as seen in \cref{eq:contourformSchutzcircle}, is that we are dealing with sums of determinants and so there is no clear way to apply the Cauchy-Binet identity (see \cite{huh2023bounded} for progress in that direction). Therefore, we tried to avoid the use of determinants and instead utilize orthogonality/duality properties for Young tableaus and Schur functions. In this section, we represent an alternative proof for the inversion in \cite{DW08} using the framework of flagged Schur and posets. \\
Instead of studying $K_{p}$, it is equivalent but more convenient to work with modified versions which do not involve cylindrical Schur functions but have a polynomial prefactor. We define for $y\preceq z\in W_{N}$, 
\begin{eqalign}\label{eq:Lambdadefinitionrealline}
\Lambda(z,y)=p_{1}^{-y_{1}}\cdots p_{N}^{-y_{N}}\sumls{T\in \T_{N}^{N, \wedge }\\\text{sh}(T)=z,\text{ledge}(T)=y  }p^{T},   \end{eqalign}
and zero otherwise (i.e. when $z_{i_{*}}<y_{i_{*}}$ for some $i_{*}\in [N]$). So the infinite matrix $\para{\Lambda(z,y)}_{(z,y)\in  W_{N,L}}$ is upper triangular. In \parencite[Proposition 2]{DW08} they show that for $y,z\in W^{N}$
\begin{eqalign}
  \Lambda(z,y)=\det\spara{h^{(jN)}_{z_{i}-y_{j}-i+j}(p)}.  
\end{eqalign}
Then in \parencite[Proposition 3]{DW08}, they prove that this matrix is invertible with inverse
\begin{eqalign}
  \Pi(y,z)=\det\spara{(-1)^{y_{i}-z_{j}-i+j}e^{(iN)}_{y_{i}-z_{j}-i+j}(p)}.  
\end{eqalign}
As described in \cite{Schurflagged}, by using the Jacobi-Trudi identity for flagged Schur for flags $f_{1}:=(2,\dotsc,N)$ and $f_{2}:=(N,\dotsc,N)$ we have
\begin{equation}\label{eq:SchurrowLambda}
\Lambda(z,y)=s_{z/y}(f_{1},f_{2};p)=\sum_{T, \operatorname{sh}(T)=z/y}p^{w(T)},    
\end{equation}
where we sum over semistandard tableaux $T$ whose shape is given by the skew diagram $z/y$ such that the entries in the ith row ,for $1\leq i\leq N-1$, must be from the set $\{i+1,\dotsc,N\}$. We can obtain \cref{eq:SchurrowLambda} directly as well from \cref{eq:Lambdadefinitionrealline}, because we observe that the ith row of $T$ has $y_{i}$-number of i's and the rest of the entries are from $\{i+1,\dotsc,N\}$. Similarly for flags $f_{3}:=(1,\dotsc,1)$ and $f_{4}:=(N-1,N-2,\dotsc,1)$ we have
\begin{equation}\label{eq:SchurcolumnPi}
\Pi(z,y)=s_{z/y}^{*}(f_{3},f_{4};-p_{N+1-\bullet})=\sum_{T, \operatorname{sh}(T)=z^{t}/y^{t}}(-p_{N+1-\bullet})^{w(T)},    
\end{equation}
where we sum over semistandard tableaux $T$ whose shape is given by the \textit{conjugate} skew diagram $z^{t}/y^{t}$ and the entries in the jth column, for $1\leq j\leq N-1$, must be from the set $\{1,\dotsc,N-j\}$ and we had to revert the weights to 
\begin{equation}
-p_{N+1-\bullet}:=(-p_{N},-p_{N-1},\dotsc,-p_{1}).    
\end{equation}
\begin{lemma}\label{lem:inversionLamdbaPirealline}
For fixed partitions $x\subset z$, we get
\begin{equation}
\sum\limits_{x\subseteq y\subseteq
z}\Lambda\left(  y,x\right)  \Pi\left(  z,y\right)  =1_{x=z}.    
\end{equation}
Namely, $\Pi$ is the inverse of $\Lambda$.
\end{lemma}
\subsection{Posets}
In this section we follow \cite{grinberg2015double}. We start by recording some of the definitions from \cite[Definition 3.1.]{grinberg2015double}.
\begin{itemize}
    \item  We shall encode \textit{posets} as pairs $(E, <)$, where $E$ is a set and
$<$ is a strict partial order (i.e., an irreflexive, transitive and antisymmetric
binary relation) on the set $E$; this relation $<$ will be regarded as the smaller
relation of the poset.

    \item A \textit{double poset} is defined as a triple $\bE=(E, <1, <2)$ where $E$ is a finite set and $<_1$ and $<_2$ are two strict partial orders on $E$. In our setting the double poset $\bE$ will be the collection of cells of the skew Young diagram and their ordering. 

    \item If $<$ is a strict partial order on a set $E$, and if $a$ and $b$ are two elements of $E$, then we say that $a$ and $b$ are $<-$\textit{comparable} if we have either $a < b \tor a = b \tor b < a$. A strict partial order $<$ on a set $E$ is said to be a \textit{total order} if and only if every two elements of $E$ are $<-$comparable.

    \item If $<$ is a strict partial order on a set $E$, and if $a$ and $b$ are two elements of $E$, then we say that $a$ is $<-$\textit{covered by $b$} if we have $a < b$ and there exists no $c \in E$ satisfying $a < c < b$. (For instance, if $<$ is the standard smaller relation on $\zz$, then each $i \in\zz$ is $<-$covered by $i + 1$.)

    \item A double poset $\bE=(E, <1, <2)$ is said to be \textit{tertispecial} if it satisfies the following condition: If $a$ and $b$ are two elements of $E$ such that $a$ is $<_1$-covered by $b$, then $a$ and $b$ are $<_2$-comparable.
    \item We fix double poset $\bE=(E, <_1, <_2)$. Then we use $\mathrm{Adm}\bE$ to denote all pairs $(P,Q)$ of subsets of $E$ such that a)$P\cup Q= E$, b)$P\cap Q=\varnothing$ and c) no $p \in P$ and $q \in Q$ satisfy $q <_1 p$. 
    
    \item Fix map $\pi : E\to \set{1, 2, 3, ...}$, then we use $\mathrm{Adm}(\bE,\pi)$ to denote all pairs $(P,Q)\in \mathrm{Adm}\bE$ such that $\pi|_P$ is a $(P,>_1 ,<_2)$-partition and $\pi|_Q$ is a $(Q,<_1 ,<_2)$-partition.

\item If $<$ is a binary relation on a set $E$, then the \textit{opposite relation} of $<$ is defined to be the binary relation $>$ on the set $E$ that is defined as follows: For any $e \in E$ and $f \in E$, we have $e > f$ if and only if $f < e$. Notice that if $<$ is a strict partial order, then so is the opposite relation $>$ of $<$.

\item If $\bE=(E, <1, <2)$ is a double poset, then an \textit{$\bE$-partition} shall mean a map $\phi : E \to \set{1, 2, 3, . . .}$ such that
\begin{itemize}
    \item every $e \in E$ and $f \in E$ satisfying $e <_1 f$ satisfy $\phi(e) \leq \phi (f)$;

    \item  every $e \in E$ and $f \in E$ satisfying $e <_1 f$ and $f <_{2} e$ satisfy $\phi(e) < \phi (f)$.
\end{itemize}

\end{itemize}
We will use \cite[lemma 6.5]{grinberg2015double}.
\begin{lemma}\label{lemma65dg15}
Let $\bE = (E, <_1, <_2)$ be a tertispecial double poset satisfying $\abs{E} > 0$.
Let $\pi : E\to \set{1, 2, 3, ...}$ be a map. Let $>_1$ denote the opposite relation of $<_1$. Then
\begin{equation}
\sum\limits_{(P,Q)\in \mathrm{Adm}(E,\pi)}(-1)^{\abs{P}}=0.    
\end{equation}
    
\end{lemma}

\subsection{{Proof of \cref{lem:inversionLamdbaPirealline} using posets}}
\begin{proof}We follow the argument described in \cite{478012}. We will consider \textit{co-semistandard} tableaux (coSSYT) that has rows that are left to right weakly increasing (nondecreasing) sequences and its columns are top to bottom strictly increasing sequences. The \textit{anti-cosemistandard} tableau (anti-coSSYT) has entries that strictly decrease along rows and weakly increase
along columns. In this language, we observe that 
\begin{equation}\label{eq.darij2.1old}
\Pi\left(  z,y\right)  =\sum_{\mathrm{coSSYT}~T}\left(
-p_{N+1-\bullet}\right)  ^{w\left(  T\right)  },
\end{equation}
where we sum over all coSSYTs $T$ of shape $z/y$ and all entries in row $i$ of $T$ are less or equal to $N-i$. By making the substitution that replaces every entry $k$ of the coSSYT by $N+1-k$, we obtain 
\begin{equation}\label{eq.darij2.1new}
\Pi\left(  z,y\right)  =\sum_{\mathrm{anti-coSSYT}~S}\left(
-p\right)  ^{w\left(  S\right)  },
\end{equation}
where we sum over all anti-coSSYTs $S$ of shape $z/y$ and all entries in row $i$ of $S$ are strictly greater than $i$. This substitution has the effect of turning coSSYTs into anti-coSSYTs and
vice versa, and also sends $\left(  -p_{N+1-\bullet}\right)  ^{w\left(
T\right)  }$ to $\left(  -p\right)  ^{w\left(  S\right)  }$ because the
exponent on $-p_{N+1-k}$ becomes the exponent on $-p_{k}$.\\
Furthermore, the addend $\left(  -p\right)  ^{w\left(  S\right)  }$ in \eqref{eq.darij2.1new} can be rewritten as $\left(  -1\right)  ^{\left\vert z/y\right\vert
}p^{w\left(  S\right)  }$ (since $\left\vert w\left(  S\right)  \right\vert
=\left\vert z/y\right\vert $). Thus, \eqref{eq.darij2.1new} rewrites as
\begin{equation}\label{eq.darij2.1newest}
\Pi\left(  z,y\right)  =\left(  -1\right)  ^{\left\vert z/y\right\vert }
\sum_{\mathrm{anti-coSSYT}~S}p^{w\left(  S\right)
},
\end{equation}
where $S$ is an anti-coSSYT of shape $z/y$ all entries in row $i$ of $S$ are strictly greater than $i$. We similarly have for $\Lambda$
\begin{equation}\label{eq.darij2.2}
\Lambda\left(  y,x\right)  =\sum_{\mathrm{SSYT}~T}p^{w\left(  T\right)  },
\end{equation}
where we sum over semistandard tableau (SSYT) $T$ of shape $y/x$ such that all entries in row $i$ of $T$ are strictly greater than $i$. Therefore, we can combine the product into single sum
\begin{eqalign}\label{eq.darij2.4}
\Lambda\left(  y,x\right)\Pi\left(  z,y\right)=&\left(  -1\right)^{\left\vert z/y\right\vert }\sum_{\mathrm{anti-coSSYT}~S}~~\sum_{\mathrm{SSYT}~T}p^{w\left(  T\right) +w\left(  S\right)}\\
=&\left(  -1\right)  ^{\left\vert z/y\right\vert }\sum_{R}p^{w(R)},    
\end{eqalign}
where we sum over \textit{broken tableau} $R$ of shape $z/y/x$, that is a tableau of shape $z/x$ whose restriction to $z/y$ is an anti-coSSYT whereas
its restriction to $y/x$ is an SSYT, and all entries in row $i$ of R are strictly greater than $i$. The broken tableaux $R$ of shape
$z/y/x$ are in bijection with the pairs $\left(  T,P\right)  $ of an SSYT $T$
of shape $y/x$ and an anti-coSSYT $P$ of shape $z/y$ (and the bijection simply
sets $T:=R\mid_{y/x}$ and $P:=R\mid_{z/y}$).\\
Next we prove that for fixed partitions $x\subset z$, we get
\begin{equation}
\sum\limits_{x\subseteq y\subseteq
z}\Lambda\left(  y,x\right)  \Pi\left(  z,y\right)  =1_{x=z}.    
\end{equation}
For $x=z$ we get a single empty tableau $R$ and so it is indeed $1$. Therefore, assuming that $x\neq z$, we must show that $\sum\limits_{x\subseteq y\subseteq
z}\Lambda\left(  y,x\right)  \Pi\left(z,y\right)  =0$. Using \eqref{eq.darij2.4}, we get
\begin{align}\label{eq.darij2.5}
 \sum\limits_{x\subseteq y\subseteq
z}\Lambda\left(  y,x\right)  \Pi\left(  z,y\right)
\nonumber& =\sum\limits_{x\subseteq y\subseteq
z}\left(  -1\right)  ^{\left\vert z/y\right\vert }\sum_{R}p^{w(R)}\\
& =\sum_{\substack{R\text{ is a tableau of shape }z/x;\\\text{all entries in
row }i\text{ of }R\text{ are }>i}}\ \ \left(  \sum
\limits_{\substack{y;\\R\text{ is a broken tableau}\\\text{of shape }
z/y/x}}\left(  -1\right)  ^{\left\vert z/y\right\vert }\right)  p^{w\left(
R\right)  },
\end{align}
where in the inner sum we are going over all partitions $y$ with $x\subseteq y\subseteq
z$ such that the tableau $R$ is in fact a broken tableau of shape $z/y/x$. To prove the identity, it suffices to show that by fixing a tableau $R$ in the outer sum,  the inner sum satisfies
\begin{equation}\label{eq.darij2.6}
\sum\limits_{\substack{y;\\R\text{ is a broken tableau}\\\text{of shape
}z/y/x}}\left(  -1\right)  ^{\left\vert z/y\right\vert }=0.
\end{equation}
We will prove this by expressing it in the framework of \cref{lemma65dg15}. We define the double poset $\mathbf{E}=\left(  E,<_{1},<_{2}\right)  $ as follows:
\begin{itemize}
    \item $E$ is the set of all cells of the skew Young diagram $z/x$.

    \item The first partial
order $<_{1}$ is given by
\begin{align*}
\left(  a,b\right)  \ <_{1}\ \left(  c,d\right)  \ \Longleftrightarrow
\ \left(  a\leq c\text{ and }b\leq d\text{ and }\left(  a,b\right)
\neq\left(  c,d\right)  \right)  .
\end{align*}
\item The second partial order $<_{2}$ is given by
\begin{align*}
\left(  a,b\right)  \ <_{2}\ \left(  c,d\right)  \ \Longleftrightarrow
\ \left(  a\leq c\text{ and }b\geq d\text{ and }\left(  a,b\right)
\neq\left(  c,d\right)  \right)  .
\end{align*}

\end{itemize}
The $\mathbf{E}$ is a tertispecial double poset and satisfies $\left\vert
E\right\vert >0$ (since $x\neq z$). Indeed, as explained in \cite[Example 3.3]{grinberg2015double}, if we have that $\left(  a,b\right)$  is $<_{1}-$covered by $ \left(  c,d\right)$, then $\left(  a,b\right)$ is either the left or top immediate neighbour of $ \left(  c,d\right)$, and so either $\left(  c,d\right)<_{2}\left(  a,b\right)$ or $\left(  a,b\right)<_{2}\left(  c,d\right)$.\\
Let $>_{1}$ and $>_{2}$ denote the
opposite relations of $<_{1}$ and $<_{2}$, respectively. Pick an integer $M$
that is larger than each entry of the tableau $R$. Then, \cref{lemma65dg15} says that any map $\pi:E\rightarrow\left\{  1,2,3,\ldots\right\}  $
satisfies
\begin{equation}\label{eq.darij2.9}
\sum_{\substack{\left(  P,Q\right)  \in\operatorname{Adm}{\mathbf{E}}
;\\\pi\mid_{P}\text{ is a }\left(  P,>_{1},<_{2}\right)  \text{-partition;}
\\\pi\mid_{Q}\text{ is a }\left(  Q,<_{1},<_{2}\right)  \text{-partition}
}}\left(  -1\right)  ^{\left\vert P\right\vert }=0.
\end{equation}
The pairs $\left(  P,Q\right)
\in\operatorname{Adm}{\mathbf{E}}$ are in bijection with the partitions $y$
satisfying $x\subseteq y\subseteq z$. Indeed, for each such partition $y$, we get a corresponding pair $\left(  P,Q\right)  \in
\operatorname{Adm}{\mathbf{E}}$ with $P=\left\{  \text{cells of
}y/x\right\}  $ and $Q=\left\{  \text{cells of }z/y\right\}  $. Thus, we can express it as
\begin{equation}\label{eq.darij2.9b}
\eqref{eq.darij2.9}=\sum_{\substack{y;\\\pi\mid_{y/x}\text{ is a }\left(  y/x,>_{1},<_{2}\right)
\text{-partition;}\\\pi\mid_{z/y}\text{ is a }\left(  z/y,<_{1},<_{2}\right)
\text{-partition}}}\left(  -1\right)  ^{\left\vert y/x\right\vert
}=0,
\end{equation}
where the sum ranges over all partitions $y$ satisfying $x\subseteq y\subseteq
z$, and where we write $y/x$ and $z/y$ for the sets $\left\{  \text{cells of
}y/x\right\}  $ and $\left\{  \text{cells of }z/y\right\}  $, respectively. We can also reverse the arrows
\begin{equation}\label{eq.darij2.9c}
\eqref{eq.darij2.9b}=\sum_{\substack{y;\\\pi\mid_{y/x}\text{ is a }\left(  y/x,<_{1},>_{2}\right)
\text{-partition;}\\\pi\mid_{z/y}\text{ is a }\left(  z/y,>_{1},>_{2}\right)
\text{-partition}}}\left(  -1\right)  ^{\left\vert y/x\right\vert
}=0,
\end{equation}
by defining $M$ to be the largest value of $\pi$ and then replacing each value $k$ of $\pi$ by $M+1-k$ (so that all the inequalities in the conditions "$\pi\mid_{y/x}$ is a $\left(  y/x,>_{1},<_{2}\right)  $-partition" and "$\pi\mid_{z/y}$ is a $\left(  z/y,<_{1},<_{2}\right)  $-partition" get reversed). 

The tableau $R$ is a map from $E$ to $\left\{  1,2,3,\ldots\right\}  $ by assigning a number entry to each cell coordinate $(a,b)$. Thus,
we can apply \eqref{eq.darij2.9c} to $\pi=R$, and obtain
\begin{equation}\label{eq.darij2.9d}
\sum_{\substack{y;\\R\mid_{y/x}\text{ is a }\left(  y/x,<_{1},>_{2}\right)
\text{-partition;}\\R\mid_{z/y}\text{ is a }\left(  z/y,>_{1},>_{2}\right)
\text{-partition}}}\left(  -1\right)  ^{\left\vert y/x\right\vert
}=0.
\end{equation}
The condition "$R\mid_{y/x}$ is a $\left(  y/x,<_{1},>_{2}\right)
$-partition" under the summation sign here is saying precisely that the
restriction $R\mid_{y/x}$ is an SSYT of shape $y/x$. Likewise, the condition
"$R\mid_{z/y}$ is a $\left(  z/y,>_{1},>_{2}\right)  $-partition" is saying
precisely that the restriction $R\mid_{z/y}$ is an anti-coSSYT of shape $z/y$.
Thus, the two conditions taken together are saying that $R$ is a broken
tableau of shape $z/y/x$. In light of this, we can rewrite the sum on the left
hand side of \eqref{eq.darij2.9d} as $\sum\limits_{\substack{y;\\R\text{ is a
broken tableau}\\\text{of shape }z/y/x}}\left(  -1\right)  ^{\left\vert
y/x\right\vert }$. Thus, \eqref{eq.darij2.9d} rewrites as
\begin{equation}\label{eq.darij2.10}
\sum\limits_{\substack{y;\\R\text{ is a broken tableau}\\\text{of shape
}z/y/x}}\left(  -1\right)  ^{\left\vert y/x\right\vert }
=0.
\end{equation}
This is almost the desired equality \eqref{eq.darij2.6}. The only difference
is that the addend is $\left(  -1\right)  ^{\left\vert y/x\right\vert }$
instead of $\left(  -1\right)  ^{\left\vert z/y\right\vert }$. But this makes
little difference, since $\left\vert z/y\right\vert +\left\vert y/x\right\vert
=\left\vert z/x\right\vert $ and thus $\left\vert y/x\right\vert =\left\vert
z/x\right\vert -\left\vert z/y\right\vert $, so that $\left(  -1\right)
^{\left\vert y/x\right\vert }=\left(  -1\right)  ^{\left\vert z/x\right\vert
-\left\vert z/y\right\vert }=\left(  -1\right)  ^{\left\vert z/y\right\vert
}\left(  -1\right)  ^{\left\vert z/x\right\vert }$, and thus we can rewrite
\eqref{eq.darij2.10} as
\begin{align*}
\sum\limits_{\substack{y;\\R\text{ is a broken tableau}\\\text{of shape
}z/y/x}}\left(  -1\right)  ^{\left\vert z/y\right\vert }\left(  -1\right)
^{\left\vert z/x\right\vert }=0.
\end{align*}
Factoring the constant sign $\left(  -1\right)  ^{\left\vert z/x\right\vert }$
out of here, we arrive precisely at \eqref{eq.darij2.6}, as desired.

\end{proof}

\appendix 
 \section{Limit to continuous time}
As mentioned in \parencite{DW08}, the "Case B: Bernoulli jumps with blocking" corresponds to the continuous time $\ZTASEP$ and they obtain the formula in \parencite{schutz1997exact}
\begin{equation} 
P(X(t)=x|X(0)=y)= \det\spara{F_{i-j}(x_{i}-y_{j},t)  }_{1\leq i,j\leq N},    
\end{equation}
where
\begin{eqalign}
F_{p}(n;t):=&e^{-t}\sumls{k=0}^{\infty}\binom{k+p-1}{p-1}\frac{t^{k+n}}{(k+n)!}\\
=&\branchmat{e^{-t}\sumls{k=0}^{\infty}\binom{k+p-1}{p-1}\frac{t^{k+n}}{(k+n)!}&\tcwhen p>0\\ e^{-t}\sumls{k=0}^{\abs{p}}(-1)^{k}\binom{\abs{p}}{k}\frac{t^{k+n}}{(k+n)!} &\tcwhen p\leq 0},    
\end{eqalign}
and $\binom{a}{b}=\frac{\Gamma(a+1)}{\Gamma(b+1)\Gamma(a-b+1)}$. 
\begin{proposition}\label{prop:continuoustime}
We will show that for equal rates $a_{k}=a=1$ and initial data $y_{i}=i$
\begin{eqalign}
&\prod_{k=1}^{N}\para{1-\frac{a_{k}}{M}}^{\floor{Mt}}\para{\frac{\frac{a_{k}}{M}}{1-\frac{a_{k}}{M}}}^{x_{k}-y_{k}} \det\spara{v_{\floor{Mt}}^{(i,j)}(x_{i}-y_{j}-i+j) },    
\end{eqalign}
where 
\begin{eqalign}
\large{v_{\floor{Mt}}^{(i,j)}(x-i+j):=\branchmat{ \sum^{i-j}_{l=0} (-1)^{l}\binom{\floor{Mt}}{x+j-i+l}\ind{0\leq x+j-i+l\leq \floor{Mt}}\suml{i<k_{1}<...<k_{l}<j+1}\frac{\frac{a_{k_{1}}}{M}}{1-\frac{a_{k_{1}}}{M}}\cdots \frac{\frac{a_{k_{l}}}{M}}{1-\frac{a_{k_{l}}}{M}}  &\tcwhen j\geq i\\ \sum^{\infty}_{l=0} \binom{\floor{Mt}}{x+j-i+l}\ind{0\leq x+j-i+l}\sumls{i<k_{m}<j+1\\\sum^{N}_{m=1} k_{m}=l}\prod^{N}_{m=1}(\frac{\frac{a_{k_{m}}}{M}}{1-\frac{a_{k_{m}}}{M} })^{k_{m}}&\tcwhen j< i}} ,   
\end{eqalign}
converges as $M\to +\infty$ to the transition density
\begin{eqalign}
\det\spara{F_{i-j}(x_{i}-y_{j},t)  }_{1\leq i,j\leq N}.   
\end{eqalign}
 
\end{proposition}
\begin{proofs}
For the first prefactor we simply have have $(1-\frac{a_{k}}{M})^{\floor{Mt}}\to e^{-a_{k}t}.$ We split cases depending on the sign of $p=i-j$.
\proofparagraph{Case $j\leq i$}
For a matrix entry with $j\leq i$ we have for $x:=x_{i}-y_{j}$
\begin{eqalign}
&\sum^{i-j}_{l=0} (-1)^{l}\binom{\floor{Mt}}{x+j-i+l}\ind{0\leq x+j-i+l\leq \floor{Mt}}\suml{i<k_{1}<...<k_{l}<j+1}\frac{\frac{a_{k_{1}}}{M}}{1-\frac{a_{k_{1}}}{M}}\cdots \frac{\frac{a_{k_{l}}}{M}}{1-\frac{a_{k_{l}}}{M}}.
\end{eqalign}
For fixed $t,x$ and large enough $M$ we have $ x+j-i+l\leq \floor{Mt}$ and also $1-\frac{a_{k_{l}}}{M}\approx 1$ and so 
\begin{eqalign}
&\approx \sum^{i-j}_{l=0} (-1)^{l}\binom{\floor{Mt}}{x+j-i+l}\suml{j<k_{1}<...< k_{l}<i+1}\frac{1}{M^{l}}a_{k_{1}}\cdots a_{k_{l}}.
\end{eqalign}
To have it match with \Schutz formula we set $a_{k}=a$ and we also bring inside the factor  $(\frac{a}{M})^{x_{i}-y_{i}}$ into the ith row.
\begin{eqalign}
&\sum^{i-j}_{l=0}\para{\frac{a}{M}}^{x_{i}-y_{i}+\ell} (-1)^{l}\binom{\floor{Mt}}{x+j-i+l}\suml{j<k_{1}<...< k_{l}<i+1}1.
\end{eqalign}
For the binomial term we use Stirling's formula
\begin{eqalign}\label{eq:binoSt}
&\binom{\floor{Mt}}{x+j-i+l}=\frac{(\floor{Mt})!}{(\floor{Mt}-(x+j-i+l))!(x+j-i+l)!}\\
&\approx \frac{1}{(x+j-i+l)!} \sqrt{\frac{\floor{Mt}}{\floor{Mt}-(x+j-i+l)}}\cdot \para{\frac{\floor{Mt}}{\floor{Mt}-(x+j-i+l)}}^{\floor{Mt}}\\
&\cdot \para{\frac{\floor{Mt}-(x+j-i+l)}{e}}^{x+j-i+l}.
\end{eqalign}
The square-root factor limits to 1 and the second factor limits to $e^{-(x+j-i+l)}$
\begin{eqalign}
\eqref{eq:binoSt}&\approx \frac{1}{(x+j-i+l)!}e^{-(x+j-i+l)}\cdot((\floor{Mt}-(x+j-i+l))/e )^{(x+j-i+l)}\\
&=\frac{1}{(x+j-i+l)!}(\floor{Mt}-(x+j-i+l))^{(x+j-i+l)}.\\    
\end{eqalign}
For the summation we have
\begin{eqalign}
&\suml{j<k_{1}<...< k_{l}<i+1} 1 = \binom{i+1-j-1}{l+1-1}=\binom{i-j}{l}.    
\end{eqalign}
All together yield
\begin{eqalign}
&\sum^{i-j}_{l=0} (-1)^{l}(\frac{a}{M})^{x_{i}-y_{i}+l-(x+j-i+l)} \binom{i-j}{l}   \frac{t^{x+j-i+l}}{(x+j-i+l)!}.
\end{eqalign}
By taking $y_{j}-y_{i}=j-i$ we are left with
\begin{eqalign}
&\sum^{i-j}_{l=0} (-1)^{l} \binom{i-j}{l}   \frac{t^{x+j-i+l}}{(x+j-i+l)!}.  \end{eqalign}
\proofparagraph{Case $j\geq  i$}
For a matrix entry with $j\geq  i$ we have for $x:=x_{i}-y_{j}$
\begin{eqalign}
&\LARGE{\sum^{\infty}_{l=0} \binom{\floor{Mt}}{x+j-i+l}\ind{0\leq x+j-i+l}\sumd{i<k_{m}<j+1}{\sum^{N}_{m=1} k_{m}=l}\prod^{N}_{m=1}(\frac{\frac{a_{k_{m}}}{M}}{1-\frac{a_{k_{m}}}{M} })^{k_{m}}}.
\end{eqalign}
For fixed t,x and large enough M we have $1-\frac{a_{k_{l}}}{M}\approx 1$ and so 
\begin{eqalign}
&\LARGE{\approx \sum^{i-j}_{l=0} \binom{\floor{Mt}}{x+j-i+l}\sumd{i<k_{m}<j+1}{\sum^{N}_{m=1} k_{m}=l}\prod^{N}_{m=1}(\frac{a_{k_{m}} }{M})^{k_{m}}}.
\end{eqalign}
To have it match with \Schutz formula we set $a_{k}=a$ and we also bring inside the factor  $(\frac{a}{M})^{x_{i}-y_{i}}$ into the ith row.
\begin{eqalign}
&\LARGE{(\frac{a}{M})^{x_{i}-y_{i}}\sum^{i-j}_{l=0} \binom{\floor{Mt}}{x+j-i+l}\sumd{i<k_{m}<j+1}{\sum^{N}_{m=1} k_{m}=l}(\frac{a}{M})^{l}}.    
\end{eqalign}
As above for the binomial term we use Stirling's formula
\begin{eqalign}
&\LARGE{\binom{\floor{Mt}}{x+j-i+l}\approx \frac{1}{(x+j-i+l)!}\frac{1}{(\floor{Mt}-(x+j-i+l))^{-(x+j-i+l)}}}.
\end{eqalign}
For the sum term we have
\begin{eqalign}
&\LARGE{\sumd{i<k_{m}<j+1}{\sum^{N}_{m=1} k_{m}=l} 1 = \binom{l+j-i-1}{j-i-1}}.    
\end{eqalign}
All together yield
\begin{eqalign}
&\sum^{\infty}_{l=0}(\frac{a}{M})^{x_{i}-y_{i}+l-(x+j-i+l)} \binom{l+j-i-1}{j-i-1} \frac{t^{x+j-i+l}}{(x+j-i+l)!}.
\end{eqalign}
By taking $y_{j}-y_{i}=j-i$ we are left with
\begin{eqalign}
&\sum^{\infty}_{l=0}  \binom{l+j-i-1}{j-i-1} \frac{t^{x+j-i+l}}{(x+j-i+l)!}.    
\end{eqalign}
   
\end{proofs}

\newpage
\section{List of further directions}\label{furtherresearchdirections}\label{part:researpappend}

\begin{researchproblem}
The natural extension is to obtain the cylindrical analogous formulas for each of the cases described in \cite{DW08}.   
\end{researchproblem}
~\\
\begin{researchproblem}
In \cite{iwao2023free}, they used free-fermionic presentations to obtain transition formulas similar to \cite{DW08}. As suggested by one of the authors, K.Motegi, the duality relation in \cite[Theorem 3.2]{iwao2023freeSchur} is the analogue of the inversion identity $\Lambda\ast \Pi(\lambda,\mu)=1_{\lambda=\mu}$ in \cite[Proposition 3]{DW08}. So there is a natural question of finding a cylindrical version of the Boson-Fermion correspondence. 
\end{researchproblem}
~\\
\begin{researchproblem}
Another approach in proving the identity $\Lambda\ast \Pi(\lambda,\mu)=1_{\lambda=\mu}$ was trying to reprove \cite[Proposition 3]{DW08} using techniques from Schubert calculus \cite{lam2014k} because Schubert calculus has a natural cylindrical analogue described in \cite{postnikov2005affine}.
\end{researchproblem}
~\\
\begin{researchproblem}
Interestingly, in \parencite{hamel1995lattice}, they define the \textit{row-flagged supersymmetric Schur function} as
\begin{eqalign}
s_{\lambda/\nu}(a,b,c,d; x,y)=\sum_{\nu}s_{\lambda/\nu}^{*}(c,d; y)s_{\nu/\mu}(a,b;x),   
\end{eqalign}
and prove that
\begin{eqalign}
s_{\lambda/\nu}(a,b,c,d; x,y)=\det\para{h_{\lambda_{i}-\mu_{j}-i+j} (x_{a_{j}},...,x_{b_{i}})/(y_{c_{j}},...,y_{d_{i}})}.
\end{eqalign}
So naturally, one can try to extend those notions to the cylinder case and thus make sense of 
\begin{eqalign}
Q_{n}=\Pi_{p}\ast P_{n}\ast K_{p}
\end{eqalign}
as a cylindrical supersymmetric Schur function. 
\end{researchproblem}

\begin{researchproblem}
In the spirit of \cite{borodin2007periodic}, it will be interesting to construct the corresponding periodic Schur process corresponding to this current tableau where $\lambda_{1}\leq \lambda_{N}+L-N$. 
\end{researchproblem}
\subsection{Does periodic TASEP have a Pfaffian formula?}
In the recent article \parencite{huh2023bounded}, they study similar iterated sums satisfying $k_1 + \dots + k_m = 0$ and write them in terms of Pfaffians in terms of some matrices
\begin{equation}\label{eq:Pfaffianformulaforsumoverzerok}
\sum_{\lambda \in \Par(m,w)}
\underset{k_1 + \dots + k_m = 0}{\sum_{k_1, \dots, k_m \in \Z}} 
 u(\lambda) \det_{1 \le i, j \le m} \left( e_{\lambda_i-i+j+Nk_i}(\vx) \right)= \Pf \left( T_p A T_p^t \right)   
\end{equation}
where $e_n(x)$ is the elementary symmetric polynomial, the set $\Par(m,w),n,m\in \mathbb{N}$ denotes
the set of partitions of length at most $m$ satisfying $\lambda_1 -
\lambda_m \le w$,  the $u(\lambda)$ is a generic function of partition $\lambda$ and $T_{p}$ is diagonal-block matrix and $A=A(u)$ is a skew-symmetric matrix. One challenge in applying this formula in the setting of \cref{eq:contourformSchutzcircle} was that there the input $x$ also depends on $Lk_{i}$ whereas in \cref{eq:Pfaffianformulaforsumoverzerok} they consider infinite tupple $\vx$, and so their bijection procedure is only focused on the index $\lambda_i-i+j+Nk_i$.

\begin{researchproblem}
Do the techniques in \parencite{huh2023bounded} generalize to the current setting to provide Pfaffian formulas for each of the three operators in $Q_{n}=\Pi_{p}\ast P_{n}\ast K_{p}$?
 
\end{researchproblem}
However, \cref{eq:Pfaffianformulaforsumoverzerok} can be useful in proving the inversion identity and evaluating to evaluating $Q_{n}=\Pi_{p}\ast P_{n}\ast K_{p}$ because here the convolution is over cylindrical partitions and operators over cylindrical diagrams.
\begin{researchproblem}
Starting from the formula
\begin{eqalign}
Q_{n}=\Pi_{p}\ast P_{n}\ast K_{p},
\end{eqalign}
does it make sense to try to obtain a biorthogonalization expression for each of the three operators and then take separate limits? Since there is a version of \textit{cylindrical Cauchy Binet } \cref{eq:Pfaffianformulaforsumoverzerok}, there is some hope that there is a cylindrical version of biorthogonalization as done in \cite[Proposition 2.10]{johansson2006random}.
\end{researchproblem}

\section{Challenges and attempts in applying the framework \parencite{BFPS} to the formula \cref{eq:contourformSchutzcircle}}\label{sec:frameworkchallenges}
\subsection{Lack of upper triangular structure}
Since we see the same kernel as in $\ZTASEP$, one can naturally directly try the framework in \parencite{BFPS}. The identity 
\begin{equation}\label{eq:identitycontourone}
F_{n+1}(x,t)=\sum_{y\geq x}F_{n}(y,t)    
\end{equation}
was used to prove \parencite[Lemma 3.2]{BFPS} where they express the transition formula in terms of Gelfand-Tsetlin patterns (GTP)
\begin{lemma}\label{lemma32BFPS}
We have
\begin{equation}
P(X(t)=x|X(0)=y)=\det[(F_{i-j}(x_{N+1-i}-y_{j};t))_{1\leq i,j\leq N}] =\sum_{\mathcal{D}}\det[(F_{-j}(x_{i+1}^{N}-y_{N-j};t))_{0\leq i,j\leq N-1}]   ,    
\end{equation}
where the sum is over the following set
\begin{equation}
 \mathcal{D}=\set{\para{x_{i}^{j}}_{2\leq i<j\leq N}: x_{i}^{j}>x_{i}^{j+1},x_{i}^{j}\geq x_{i-1}^{j-1}},   
\end{equation}
and $x^{j}_{1}:=x_{k},$ for $k=1,...,N$.
\end{lemma}
This lemma is key for several reasons. One is that the indicator over the set $\mathcal{D}$ can be expressed in terms of a product of determinants and thus express the transition formula as a determinantal point measure. Second, it turns
\begin{equation}
F_{i-j}\mapsto     F_{-j}= \oint_{|w|=R}e^{t(w-1)}\frac{(w-1)^{j}}{w^{x-y_{N+1-j}+j+1}}\frac{dw}{2\pi i},
\end{equation}
where now this contour no longer has a pole at $w=1$ and so one obtains yet another identity
\begin{equation}\label{eq:identitycontourtwo}
F_{n+1}(x,t)=-\sum_{y< x}F_{n}(y,t).       
\end{equation}
This identity is the key reason that then a certain matrix-system becomes upper triangular and thus giving existence of a biorthogonal representation for the kernel.\\
In \cref{eq:contourformSchutzcircle}, the \cref{eq:identitycontourone} is still true. So one can again obtain a set $\mathcal{D}$. However, now the mapping is
\begin{equation}
F_{i-j-Nk_{i}}\mapsto     F_{-j-Nk_{i}}= \oint_{|w|=R}e^{t(w-1)}\frac{(w-1)^{j+Nk_{i}}}{w^{x-y_{N+1-j}+j+Nk_{i}+1}}\frac{dw}{2\pi i}.
\end{equation}
So when $k_{i}<0$, this contour will still have a pole at $w=1$ and so the \cref{eq:identitycontourtwo} is no longer true and thus the corresponding matrix is no longer upper triangular. From here one could try multiple ideas. One could keep applying \cref{eq:identitycontourone} to remove all $k_{i}<0$ (and even the $k_i>0$ using $ F_{n}(x)=\sum_{y\in [x,x+1]} (-1)^{x-y}F_{n+1}(y)$) but then there is no Gelfand-Tsetlin pattern and thus no clear way to write the transition as a determinantal point measure.
\begin{researchproblem}
Does \cref{eq:contourformSchutzcircle} have a representation in terms of some type of Gelfand-Tsetlin patterns?
\end{researchproblem}

\subsection{Attempting to reduce to finite sums}
A key challenge in taking limits of the \cref{eq:contourformSchutzcircle} is that we are getting infinitely many sums for $k\in \mathbb{Z}^{\infty}$. In the following lemma \cite[lemma 8]{povolotsky2007determinant}, they showed that the contour simplifies the range of the vector $k\in \mathbb{Z}^{N}$. 
\begin{lemma}\label{tuplereductionPriezzhev}
Let $(x_{1},...,x_{N}),(y_{1},...,y_{N})$ satisfy $1\leq a_{1}\leq...\leq a_{N}\leq L$, then  the necessary conditions for the product
\begin{equation}\label{tuplereductionPriezzhevformula}
\prod_{i=1}^{N}\int_{\Gamma^{\infty}}\para{\frac{1+\lambda z_{i}}{1-z_{i}^{-1}}  }^{i-\sigma(i)+Nk_{i}-\sum_{j=1}k_{j}}z_{i}^{-x_{i}+y_{\sigma(i)-Lk_{i}}}\frac{\dz_{i}}{2\pi i z_{i}}    
\end{equation}
to be non-zero are $k_{i}\in \set{-1,0,1 }$.
\end{lemma}
Since the kernel in \cref{eq:contourformSchutzcircle} differs from \cref{tuplereductionPriezzhevformula} by containing an exponential term $e^{t(w-1)}$, we attempted to remove it. Here is a lemma that might have independent interest. In the spirit of writing the transition formula in terms of path-decomposition, we found the following representation of the kernel $e^{-t(I+\nabla^{-})}(x-y)$.
\begin{lemma}\label{Rtsplittingscheme}
For
\begin{eqalign}
R_{t}(x-y)=e^{-t(I+\nabla^{-})}(x-y)=e^{-t}\frac{t^{x-y}}{(x-y)!}\ind{x\geq y}=\frac{1}{2\pi i}\oint_{\Gamma_{0}}\frac{e^{t(w-1)}}{w^{x-y+1}} \dw,    
\end{eqalign}
we have the limiting formula
\begin{equation} 
R_{t}(x,y)=e^{-t}\lim_{n\to+\infty}\text{\Large$ \ast $}^{n}P_{t/n}(x,y),     
\end{equation}
where $\text{\Large$ \ast $}^{n}$ means nth-convolution and $P_{t/n}(x,y):=\left(\frac{t}{n}\right)^{x-y}\ind{y\in [x-1,x]}$, which also has the contour formula
\begin{eqalign}
\text{\Large$ \ast $}^{n}P_{t/n}(x,y)=&  \frac{1}{2\pi i}\oint_{\Gamma_{0}}\frac{(1+\frac{tw}{n})^{n}}{w^{x-y+1}}\dw=\left(\frac{t}{n}\right)^{x-y}\ind{x\geq y\geq x-n  }\para{\binom{n}{x-y}+\binom{n}{x-y-1}}.
\end{eqalign}
\end{lemma}  
 \begin{proof}
Taking the Z-transform of $R_{t}$ gives
\begin{equation}
\widehat{R_{t}}(z):=\sum_{k\in \mathbb{Z}}R_{t}(k)z^{k}=e^{tz}=  \lim_{n\to+\infty}\left(1+\frac{tz}{n} \right)^{n}.  
\end{equation}
For each factor we have
\begin{equation} 
\left(1+\frac{tz}{n} \right)=\sum_{k\in \mathbb{Z}}\ind{k=0,1}   \left(\frac{t}{n}\right)^{k}z^{k} =\widehat{P_{t/n}}(z).
\end{equation}
Therefore, by inverting we get the desired formula. We can also get this directly by evaluating the limit:
\begin{equation}
\lim_{n\to+\infty}\text{\Large$ \ast $}^{n}P_{t/n}(x,y).    
\end{equation}
We have
\begin{align}
\text{\Large$ \ast $}^{n}P_{t/n}(x,y)=&\sum  \left(\frac{t}{n}\right)^{x-z_{1}}\hspace{-0.5cm}\ind{z_{1}\in [x-1,x]}   \cdots \left(\frac{t}{n}\right)^{z_{n-1}-y}\hspace{-0.5cm}\ind{y\in [z_{n-1}-1,z_{n-1}]}\\ & =\left(\frac{t}{n}\right)^{x-y}\sum  \ind{z_{1}\in [x-1,x]}   \cdots \ind{y\in [z_{n-1}-1,z_{n-1}]}   
\end{align}
The inner most sum is
\begin{align}
&\sum  \ind{z_{n-1}\in [z_{n-2}-1,z_{n-2}]}   \ind{y\in [z_{n-1}-1,z_{n-1}]}   \\
&= \ind{y\in [z_{n-2}-1,z_{n-2}]} +\ind{y+1\in [z_{n-2}-1,z_{n-2}]}
\end{align}
and so by repeating we have
\begin{align}
\sum  \ind{z_{1}\in [x-1,x]}   \cdots \ind{y\in [z_{n-1}-1,z_{n-1}]}  &=\ind{x\geq y\geq x-n  }\sum_{k=0}^{n}\binom{n}{k} \ind{y+k\in [x-1,x]}\\
&=\ind{x\geq y\geq x-n  }\para{\binom{n}{x-y}+\binom{n}{x-y-1}}.   
\end{align}
Using the asymptotic $\lim_{n\to \infty}\frac{1}{n^{x-y}}\binom{n}{x-y}=\frac{1}{(x-y)!}$ we find the desired limit:
\begin{equation}
\left(\frac{t}{n}\right)^{x-y}\ind{x\geq y\geq x-n  }\para{\binom{n}{x-y}+\binom{n}{x-y-1}} \to \frac{t^{x-y}}{(x-y)!}\ind{x\geq y}. 
\end{equation}

\end{proof} 

\begin{researchproblem}
Is it possible to reduce the sums in \cref{eq:contourformSchutzcircle} to be over finite sets?    
\end{researchproblem}

\subsection{Fractional-windings formula}
In trying to find some version for \cref{lemma32BFPS}, we came across the following representation where now the tupple is in the set $k\in (\zz+\frac{\zz}{L})^{N}$. We tried this representation in hopes of getting more cancellations in the context of the proof of \cref{lemma32BFPS} and ideally give rise to a \textit{cylindrical Gelfand-Tsetlin pattern}.
\begin{proposition}\label{RintformulaPTASEP}
By requiring $\frac{L}{N}\in \mathbb{N}$, we obtain
\begin{equation}
\label{eq:transi_wind}
P(X(t)=x|X(0)=y)=\sum_{\substack{\sum k_{i}=0\\(k_{1},..,k_{N})\in (\zz+\frac{\zz}{L})^{N} }}det[(F_{i-j-Nk_{i}}(x_{N+1-i}-y_{j}-Lk_{i}))_{1\leq i,j\leq N}]  \frac{1}{L^{N}} ,    
\end{equation}
where the domain of each variable is $k_{i}\in \{0, \pm \frac{1}{L}, \pm \frac{2}{L},...,\pm 1,\pm 1\pm \frac{1}{L},...\}$ and as in \Schutz's formula the determinant entry in contour form is:
\begin{equation}
F_{i-j-Nk_{i}}(x_{N+1-i}-y_{N+1-j}-Lk_{i}):= \oint_{\Gamma_{outer}}e^{t(w-1)}\frac{(w-1)^{-(i-j)+Nk_{i}}}{w^{x_{N+1-i}-y_{N+1-j}-Lk_{i}-(i-j)+1+Nk_{i}}}\frac{dw}{2\pi i},
\end{equation}
with $0\leq y_{1}<y_{2}<...<y_{N}\leq L-1$ and $0\leq x_{N}<x_{N-1}<...<x_{1}\leq L-1$. 
\end{proposition}
\begin{proof}
First we note that since for $w\in R_{z}$ we have $z^{L}=p(w)\Rightarrow \frac{p(w)}{z^{L}}=1$, we can write the average sum
\begin{equation}
\frac{1}{L}\sum_{w\in R_{z}}\frac{f(w)}{q'_{z}(w)} =\frac{1}{L}\sum _{a=0}^{L-1}\frac{1}{L}\sum_{w\in R_{z}}\frac{f(w)}{q'_{z}(w)} \para{ \frac{p(w)}{z^{L} }}^{\frac{a}{L}},        
\end{equation}
for $a=0,...,L-1$. The function 
\begin{equation}
f(w) \para{ \frac{p(w)}{z^{L}}}^{\frac{a}{L}} =z^{\frac{-aL}{N}}(w-1)^{j-i+N+\frac{N}{L}a}w^{-x_{i}+y_{j}+i-j+L-N-1 +a-\frac{N}{L}a}e^{t(w-1)}
\end{equation}
 is analytic on $\mathbb{C}\setminus [0,1]$, where we removed the branch cut $[0,1]$ for the fractional function $w^{\frac{N}{L}}(w-1)^{-\frac{N}{L}}$. To use the residue theorem
\begin{equation}
\frac{1}{L}\sum _{a=0}^{L-1}\frac{1}{L}\sum_{w\in R_{z}}\frac{f(w)}{q'_{z}(w)} \para{ \frac{p(w)}{z^{L} }}^{\frac{a}{L}}=\frac{1}{L}\sum _{a=0}^{L-1} \oint_{\Gamma_{outer}}\frac{f(w)}{p(w)-z^{L}}\para{ \frac{p(w)}{z^{L}}}^{\frac{a}{L}} \frac{dw}{2\pi i}-\oint_{\Gamma_{inner}}\frac{f(w)}{p(w)-z^{L}}\para{ \frac{p(w)}{z^{L}}}^{\frac{a}{L}} \frac{dw}{2\pi i},
\end{equation}
we need contours $\Gamma_{outer},\Gamma_{inner}$ that are around $[0,1]$ and their topological annulus D with boundaries $\partial_{inner} D:=\Gamma_{inner},\partial_{outer} D:=\Gamma_{outer}$ contains the Bethe roots of $q_{z}(w)$.  As observed in \cites{povolotsky2007determinant,baik2018fluctuations}, one possible choice is to take the level set
\begin{equation}
\Gamma_{inner}:=\{w: \abs{w^{L-N}(w-1)^{N} }=C  \},    
\end{equation}
because it becomes an oval containing $[0,1]$ when $C>(1-\frac{N}{L})^{L-N}(\frac{N}{L})^{N} $. 
So we also take $|z|>C^{\frac{1}{L}}$ in the z-contour outside the determinant.  Next we use the geometric series for each of the contours. For the first contour we take the contour $\Gamma_{outer}$ to be large enough so that $\abs{\frac{z^{L}}{p(w)}}<1$ and obtain
\begin{equation}
   \oint_{\Gamma_{outer}}\frac{f(w)}{p(w)-z^{L}}\para{ \frac{p(w)}{z^{L}}}^{\frac{a}{L}} \frac{dw}{2\pi i} =\sum_{k\geq 0}z^{L(k-\frac{a}{L})}\oint_{\Gamma_{outer}}\frac{f(w)}{p(w)}\frac{1}{p(w)^{k-\frac{a}{L}}}\frac{dw}{2\pi i}
\end{equation}
and for the second contour we use that $|z|>C^{\frac{1}{L}}\Rightarrow \abs{\frac{p(w)}{z^{L}}}<1$ to obtain
\begin{equation}
  - \oint_{\Gamma_{inner}}\frac{f(w)}{p(w)-z^{L}}\para{ \frac{p(w)}{z^{L}}}^{\frac{a}{L}} \frac{dw}{2\pi i} =\sum_{k\leq -1 }z^{L(k-\frac{a}{L})}\oint_{\Gamma_{inner}}\frac{f(w)}{p(w)}\frac{1}{p(w)^{k-\frac{a}{L}}}\frac{dw}{2\pi i}.
\end{equation}
For the second contour the integrand has no other poles outside $[0,1]$ and so we can enlarge it to the $\Gamma_{outer}$ to match the first one, so all together we find
\begin{equation}
\frac{1}{L}\sum _{a=0}^{L-1}  \sum_{k\in \zz }z^{L(k-\frac{a}{L})}\oint_{\Gamma_{outer}}\frac{f(w)}{p(w)^{k+1-\frac{a}{L}}}\frac{dw}{2\pi i}.
\end{equation}
Finally by pulling out the sums from each row, we obtain a result of the form
\begin{equation}
\frac{1}{L^{N}}\sum _{a_{1},...,a_{N}=0}^{L-1}  \sum_{k_{1},...,k_{N}\in \zz }  z^{L(\sum k_{i}-\frac{a_{i}}{L})-1} F(a,k).  
\end{equation}
Since $L\frac{a}{L }\in \mathbb{N}$, we obtain that for the z-contour to be non-zero we require $L(\sum k_{i}-\frac{a_{i}}{L})=0$. Finally, we do the change of variables $\wt{k}_{i}- k_{i}-\frac{a_{i}}{L}$ to get the summation domain $\set{ 0, \pm \frac{1}{L},... }$.

\end{proof}

\printbibliography
\end{document}